\documentclass{amsart}
\usepackage{amsfonts, amsmath,amscd, amssymb, latexsym, mathrsfs, verbatim, wasysym } %stmaryrd
\usepackage{amssymb}
\usepackage{hyperref}
\usepackage{enumitem}
\usepackage[dvips]{graphicx}
\usepackage{epsfig}
\usepackage{psfrag}
\usepackage[colorinlistoftodos]{todonotes}
\usepackage[utf8x]{inputenc}
\usepackage[all]{xy}
\usepackage{amscd}
\usepackage{tikz-cd}
\usepackage{appendix}
\usepackage{amsthm}
\usepackage{appendix}
\usepackage{mathtools}
\usepackage[dvips]{graphicx}
\usepackage{amsfonts}
\usepackage{amssymb}
\usepackage{epsfig}
\usepackage{psfrag}
\usepackage{subfigure}
\usepackage{captcont}
\usepackage{color}
\definecolor{darkgreen}{cmyk}{1,0,1,.2}
\definecolor{m}{rgb}{1,0.1,1}
\definecolor{green}{cmyk}{1,0,1,0}
\definecolor{darkred}{rgb}{0.55, 0.0, 0.0}
\definecolor{test}{rgb}{1,0,0}
\definecolor{cmyk}{cmyk}{0,1,1,0}

\newcounter{diagram}
\numberwithin{diagram}{section}
\numberwithin{equation}{section}

\newtheorem{Equation}{}[section]

\newtheorem{theorem}[Equation]{Theorem}
\newtheorem{proposition}[Equation]{Proposition}
\newtheorem{lemma}[Equation]{Lemma}
\newtheorem{corollary}[Equation]{Corollary}

\newtheorem{definition}[Equation]{Definition}

\newtheorem{remark}[Equation]{Remark}
\newtheorem{notations}[Equation]{Notations}

\def\ev{\operatorname{ev}}
\def\diag{\operatorname{diag}}
\def\Ext{\operatorname{Ext}}

\def\span{\operatorname{span}}

\def\Supp{\operatorname{Supp}}

\def\diam{\operatorname{diam}}

\def\pr{\operatorname{pr}}
\def\id{\operatorname{id}}
\def\Id{\operatorname{Id}}
\def\Range{\operatorname{Range}}
\def\supp{\operatorname{supp}}

\def\maA{\mathcal{A}}

\def\maE{\mathcal{E}}

\def\maK{\mathcal{K}}

\def\maL{\mathcal{L}}

\def\maP{\mathcal{P}}

\def\C{\mathbb C}

\def\R{\mathbb R}

\def\N{\mathbb N}

\def\maA{{\mathcal A}}

\def\maE{{\mathcal E}}

\def\what{\widehat}

\def\ep{\epsilon}

\usepackage{color}
\definecolor{darkgreen}{cmyk}{1,0,1,.2}
\definecolor{m}{rgb}{1,0.1,1}
\definecolor{green}{cmyk}{1,0,1,0}

\definecolor{test}{rgb}{1,0,0}
\definecolor{cmyk}{cmyk}{0,1,1,0}

\begin{document}

\title[Equivariant PPV and Paschke duality]{An equivariant  PPV theorem and \\ Paschke-Higson duality}

 %\author{Moulay-Tahar Benameur and Indrava Roy}
\author[M-T. Benameur]{Moulay-Tahar Benameur}
\address{IMAG, Univ. Montpellier, CNRS, Montpellier, France}
\email{moulay.benameur@umontpellier.fr}

\author[I. Roy]{Indrava Roy}
\address{Institute of Mathematical Sciences (HBNI), Chennai, India}
\email{indrava@imsc.res.in}

\maketitle

\begin{abstract}
We state the Paschke-Higson duality theorem for a transformation groupoid. Our proof relies on an equivariant localized and norm-controlled version of the Pimsner-Popa-Voiculescu theorem. The main consequence is the existence of a Higson-Roe exact sequence, involving the Baum-Connes assembly map for such groupoid.
\end{abstract}

\bigskip

\tableofcontents

\section*{Introduction}

In their groundbreaking work on coarse geometry and the Baum-Connes conjecture, Higson and Roe established the existence of a long exact sequence in $K$-theory, now called the Higson-Roe sequence. Informally, this sequence ``measures" the potential failure of the Baum-Connes conjecture, i.e. the failure of the Baum-Connes assembly map to be an isomorphism \cite{HR-I, HR-II, HR-III}. The Higson-Roe sequence uses the Paschke-Higson duality isomorphism in a fundamental way, which allows  to express the $K$-homology of a compact space $Y$ as the $K$-theory of some dual $C^*$-algebra $Q(Y, H)$, which can be taken to be the commutant of $C(Y)$  in the Calkin algebra of any ample representation of $C(Y)$, see  \cite{Paschke81, HigsonPaschke}. More generally, Paschke showed in \cite{Paschke81} that for any
separable unital $C^*$-algebra $A$, the $K$-theory of the commutant of $A$ in the Calkin algebra (for any
ample representation)  is isomorphic to the group of invertibles in $\text{Ext}(A)$, see also  \cite{HPR, HRbook} and  \cite{Valette} for further developments.
In this context, Paschke-Higson duality as applied to the Higson-Roe sequence, relies on a remarkable and deep theorem of Voiculescu \cite{Voiculescu}. Indeed, the commutative case of Voiculescu's theorem  provides a crucial step in identifying the boundary map in the Higson-Roe exact sequence with the Baum-Connes assembly map.

Recall that Voiculescu's theorem was originally carried out in order to solve some then  open questions in operator theory \cite{Halmos}, it also implied a noncommutative version of the Brown-Douglas-Fillmore theorem about the  existence of the trivial element of  $\Ext (Y)$ \cite{BDF}, a far-reaching generalization of the classical Weyl-von Neumann classification theorem \cite{Weyl, vonNeumann, Berg}. It is worth pointing out that  Kasparov  studied, in the early eighties,
representations of unital, nuclear $C^*$-algebras on Hilbert {{$C^*$-modules}}, and proved a generalized Voiculescu theorem which   played  a crucial part in establishing his powerful $KK$-theory, see \cite{Kasparov, Kasparov1}.

In the present paper, we shall  be interested in the statement and applications of an equivariant family version of the Voiculescu theorem. We shall only be interested  in the Paschke-Higson duality corollary and the resulting Higson-Roe exact sequence for transformation groupoids. Other applications will be investigated elsewhere.
The equivariant Paschke-Higson duality   includes a proper and cocompact  action of a countable discrete group $\Gamma$,  and allows  to express the classical Baum-Connes assembly map \cite{BaumConnes} using boundary maps  in $K$-theory of $C^*$-algebras, thereby providing a $K$-theory  obstruction group. This latter group turns out to be a natural receptacle   for some defect secondary eta invariants of Dirac operators.
Paschke-Higson duality  therefore provides a bridge between equivariant $K$-homology and the $K$-theory of appropriate {\em{coarse}} algebras associated with proper, co-compact group actions on non-compact spaces, see for instance \cite{Roe1}.
In \cite{HigsonRoe2008}, this approach allowed to give a nice proof of the Keswani rigidity  of reduced eta invariants, for the spin Dirac operator in the presence of  positive scalar curvature metrics and for the signature operator in the presence of a homotopy equivalence. Other results were obtained following the same line of ideas,  in relation with  the Novikov and Gromov-Lawson-Rosenberg conjectures, see for instance  \cite{HR-I, HR-II, HR-III,  BenameurMathaiSections, BRJFA, BEKW,  PS, XY, Zenobi, Zeidler} as well as some slightly different approaches in \cite{HPS, STY, Yu1, Yu2}.

 So our {first goal  is} to prove the {\em{equivariant family}} version of the Paschke-Higson duality theorem, which contains as a special case the Higson-Roe equivariant version recalled  above. The countable discrete group $\Gamma$ must now be replaced by a transformation groupoid $X\rtimes\Gamma$ and, guided by the family version and the equivariant version as well as the Baum-Connes assembly map for \'etale groupoids, we had to provide  for any proper cocompact $\Gamma$-space $Z$, some dual $C^*$-algebra $Q (X, \Gamma, Z, H)$ associated with an ``ample'' representation $H$ whose $K$-theory is intended to be isomorphic to the $\Gamma$-equivariant bivariant group $KK_\Gamma^* (Z, X)$. Notice indeed that this latter group is the building block of the LHS of the Baum-Connes assembly map for the transformation groupoid $X\rtimes \Gamma$. See  Theorem \ref{Paschke3} which states the precise Paschke-Higson duality for the transformation groupoid $X\rtimes \Gamma$. When $X$ is reduced to the point, {the Paschke-Higson Theorem \ref{Paschke3} reduces} as expected to the classical Higson-Roe result \cite{HigsonRoe2008, BRJFA}. When the group is trivial and all involved spaces are compact, one recovers the Pimsner-Popa-Voiculescu (PPV) setting and the Paschke-Higson result  is already stated  in a different form by Connes and Skandalis in \cite{ConnesSkandalis} using $KK$-theory.  The dual algebra $Q (X, \Gamma, Z, H)$ is more precisely constructed using a generalized version of the Roe algebras generated by finite propagation operators, and we had to precisely keep track of the finite propagation properties of the intertwining unitaries appearing in the PPV work and also to take care of the extra equivariance properties. With the extra action of the countable discrete group, we show that  one can indeed ensure the localization of the supports of these intertwining unitaries, a notion that generalizes the finite propagation property to the non-cocompact case, see Definition  \ref{LocalizedOp}. The invariance of the unitaries is fulfilled by using standard averagings, which in turn only converge in the strong topology  and hence tend to violate  the desired intertwining up to compacts property.
%Moreover, we also state a {norm-controlled version} ensuring estimates of the  involved defect compact operators, {which were crucial for the operator theory applications  in the original work of Voiculescu.}
%More precisely, we prove an equivariant, {support-localized version of the  PPV theorem which applies in the non-cocompact case and ensures in the cocompact case} the needed {\underline{finite propagation}} of the intertwining unitaries.
 In the cocompact and metric case, we eventually obtain the needed finite propagation property and the proof of the expected  equivariant family Paschke-Higson theorem.

 As explained above, our main application of the Paschke-Higson theorem for $X\rtimes \Gamma$ is to deduce  a Higson-Roe exact sequence, encompassing the corresponding Baum-Connes assembly map for this groupoid, see the companion paper \cite{BR2} and also \cite{BR1}.  This  exact sequence yields, as in the classical case, to rigidity applications of some secondary invariants of Dirac operators on  suspended foliated spaces \cite{BenameurPiazza}, especially laminations such as the principal solenoidal tori \cite{CandelConlonVol1} used in \cite{BenameurMathaiJFA}.
 Exactly as for the classical Paschke-Higson duality,  we were naturally led to the statement of an equivariant  family version of the Voiculescu theorem, an independent  result.
  In the non-equivariant case, the family Voiculescu theorem is an extension  of a classical theorem  due to Pimsner-Popa-Voiculescu \cite{PPV}.
%  As explained in \cite{PPV}, this non-equivariant result   can already be understood as  an extension of the fundamental Weyl-von Neumann theorem and therefore has independent potential applications in operator theory.
%  Recall that in the PPV work,  the ``covariant'' variable is a commutative unital $C^*$-algebra while the ``contravariant'' variable is a noncommutative (unital) $C^*$-algebra, and a bivariant $\Ext$ theory was then proposed and expanded later on in relation with bivariant $K$-theory. \\
\\

{\em{Acknowledgements.}} The authors  wish to thank A. Carey,  T. Fack, J. Heitsch, N. Higson,   H. Oyono-Oyono, V. Nistor, M. Puschnigg,  A. Rennie, J. Rosenberg, G. Skandalis
and A. Valette for many helpful discussions. They also  express their gratitude to the referee for her/his careful reading of previous versions of the manuscript
and for her/his many helpful suggestions,  which significantly improved the paper.
MB  thanks the French National Research Agency for the financial support via the ANR-14-CE25-0012-01 (SINGSTAR). IR thanks the Homi Bhabha National Institute and the Indian Science and Engineering Research Board via MATRICS project MTR/2017/000835 for support.

\bigskip

\section{Statement of the main theorems}\label{Statements}
All the spaces considered in the present paper are assumed second-countable.
We devote this preliminary section to the detailed statement of the main results. Let $X$ be a  compact metrizable space of finite dimension, and let $\Gamma$ be a discrete infinite countable group acting
by homeomorphisms on $X$. {Given a locally compact Hausdorff space $Z$ with a proper (not necessarily cocompact) $\Gamma$-action, recall the usual $C^*$-algebra $C_0(Z)$ of continuous functions vanishing at infinity. We shall also need the following  $\Gamma$-equivariant $C^*$-ideal  $C_{0\vert\Gamma} (Z)$ of $C_b(Z)$. If $C_{c\vert\Gamma} (Z)$ denotes the $\Gamma$-algebra of those continuous bounded functions $f$ on $Z$ such that only a finite number of elements  $g\in \Gamma$ may satisfy $g\supp (f)\cap \supp (f) \neq \emptyset$, the algebra $C_{0\vert\Gamma} (Z)$ will be the closure of $C_{c\vert\Gamma} (Z)$ in $C_b(Z)$. Notice that $C_{0\vert\Gamma} (Z)$  reduces to  $\Gamma$-algebra $C_0(Z)$ when the action is cocompact (say when $Z/\Gamma$ is compact), and is thus separable in this case. In general, it will obviously not be separable.}  Consider now a separable $\Gamma$-algebra $A$ which is a $\Gamma$-proper $C^*$-algebra over a  locally compact Hausdorff space $Z$. Recall that this means that $Z$ is a proper $\Gamma$-space in the usual sense and that there exists a $\Gamma$-equivariant  morphism $C_0(Z)\to ZM (A)$ from $C_0(Z)$ to the center $ZM(A)$ of the multiplier algebra $M(A)$ of $A$,  such that $C_0(Z) A$ is dense in $A$. We denote for simplicity by $fa\in A$ the resulting action of $f\in C_0(Z)$ on $a\in A$. The first example of such algebra $A$ is $C_0(Z)$ itself but given such $A$ for any extra separable $\Gamma$-algebra $B$, the $\Gamma$-algebra $A\otimes B$ is then again $\Gamma$-proper over $Z$.  Since we are mainly interested in examples like $C_0(Z, B)$ where $B$ is a given  separable unital $\Gamma$-algebra, we shall always assume that $C_0(Z)$ maps to the center $ZA$ of the $C^*$-algebra $A$ itself. {Notice that this can be ensured by replacing $A$ by $A+C_0(Z)$ where $C_0(Z)$ is meant as its range in the multiplier algebra $M(A)$}.

{When the $\Gamma$-space $Z$ is cocompact, it will sometimes be convenient to express the properties of our operators in terms of propagation with respect to a given proper  $\Gamma$-invariant distance $d$ on $Z$, so as to compare with the litterature. All our results are though valid without reference to such distance by using the new notion of localized operators expanded in Appendix A. }
The diagonal action of $\Gamma$ on $X\times Z$ then endows $X\times Z$ with a proper action. Let $G$ denote the transformation groupoid $X\rtimes \Gamma$. If a Hilbert space $H$ is endowed with a unitary action of $\Gamma$, then  a given $C(X)$-representation
${\what \pi}: C(X, A) \rightarrow \maL_{C(X)} (C(X)\otimes H)$ is a $G$-equivariant representation if the corresponding field $(\pi_x)_{x\in X}$ of representations of $A$
is $\Gamma$-equivariant. The same comment applies to a $G$-operator from $\maL_{C(X)} (C(X)\otimes H)$ which then corresponds to a $\Gamma$-equivariant $*$-strongly continuous field of
operators in $H$.

\subsection{An extended PPV theorem}\label{PPVextended}

We fix the proper $\Gamma$-algebra $A$ over $Z$  as above. Recall again that we have assumed that $C_0(Z)$ maps inside the center $ZA$ of $A$. The $C^*$-algebra $C(X, A)$ of continuous functions from $X$ to $A$ is  naturally equipped with a $C(X)$-algebra structure and the action of $\Gamma$ endows it with the structure of a $G$-algebra, see \cite{LeGall, BR1}. Suppose that $E$ is a countably-generated Hilbert $C(X)$-module.  We shall denote abusively by $\maL_{C(X)} (E)$ the $C^*$-algebra of adjointable operators in $E$ and by $\maK_{C(X)} (E)$ its ideal of $C(X)$-compact operators \cite{Kasparov}.

A given representation $\what\pi: C(X, A)\rightarrow \maL_{C(X)} (E)$ is called a $C(X)$-representation if the  action of $C(X)$ on $C(X, A)$ is compatible with
the right $C(X)$-module
structure on $E$. Such a  $C(X)$-representation then corresponds to  a $*$-homomorphism $\pi: A\rightarrow \maL_{C(X)} (E)$, which in turn corresponds to a field of representations $\pi_x: A\rightarrow \maL (E_x)$, where $E_x:= E\otimes_{\ev_x}\C$ is the Hilbert space
fibre over $x$ associated with the Hilbert module $E$. Recall that the field $(E_x)_{x\in X}$ is then a continuous field of Hilbert spaces in the sense of  \cite{Dixmier}. {Only the $C(X)$-algebra  $C(X, A)$ will be needed in the present paper, meaning a constant field, and we shall always use  in the sequel this notation of adding  a hat for the $C(X)$-representation of $C(X, A)$ associated with a given $*$-homomorphism from $A$ to $\maL_{C(X)} (E)$. We have chosen to state our results in this language of $C(X)$-representations for the sake of possible generalizations, see \cite{BR1, BR2}.}

Once such representation $\pi$ is fixed and {$Z$ is} metric-proper as above, an operator $T\in \maL_{C(X)} (E)$ will {be said to have} finite propagation $\leq R$ (with respect to $\pi$) if
$$
\pi (a_1) T \pi (a_2) =0 \text{ for any }a_1, a_2\in A\text{ such that } d (\Supp (a_1), \Supp (a_2)) >R.
$$
Recall that the support $\Supp (a)$ of an element $a\in A$ is the complement of the largest open subspace $U$ of $Z$ such that $fa=0$ for any $f\in C_0(U)$. {{When $Z$ is not necessarily a proper-metric space, the support $\Supp(T)$ of the operator $T\in \maL_{C(X)}(E)$ itself with respect to the representation $\pi$ can still be defined as the complement in $Z\times Z$ of the union of all open sets of the form $U\times V$, where $U$ and $V$ are open in $Z$, such that $\pi(a_1)T\pi(a_2)=0$ for any $a_1\in C_0(U)A$ and $a_2\in C_0(V)A$.
%The notion of finite propagation can then be replaced by the more general condition  that the operator $T$ is {\underline{localized}}, say that the subset of $\Gamma$ given by
%$$
%\{g\in \Gamma\vert (g\times \id) \Supp (T)\cap \Supp (T) \neq \emptyset\}
%$$
%is finite. As for finite propagation operators, it is easy to check that the space of localized operators is a $*$-sublagebra of the $C^*$-algebra $\maL_{C(X)} (E)$. M2I+M: IS THIS TRUE? NOT SO OBVIOUS A PRIORI? A classical argument shows that the support of the composite operator $T_1\circ T_2$ is contained in the closure of the space
%$$
%\{(z, z'')\in Z^2\vert \exists z'\in Z\text{ such that } (z, z')\in \Supp (T_1) \text{ and } (z', z'')\in \Supp (T_2)\}.
%$$
}

{\begin{definition}
A cutoff function  on the $\Gamma$-proper space $Z$ will be any continuous function $\chi: Z\rightarrow [0, 1]$ such that
\begin{enumerate}
\item $\chi$ belongs to $C_{c\vert\Gamma} (Z)$;
\item $\sum_{g\in \Gamma} g\chi =1$.
\end{enumerate}
\end{definition}
Recall that the first item means that $
\{g\in \Gamma, g\supp (\chi)\cap \supp (\chi) \neq \emptyset\}$ is finite.
The second item  means that for any $z\in Z$, $\sum_{g\in \Gamma} \chi (gz) =1$, the latter sum being finite by the first item. Moreover, if $W_\chi=\{\chi\neq 0\}$ then $Z=\bigcup_{g\in \Gamma} gW_\chi$, and for any compact subspace $K\subset Z$, the set $\{g\in \Gamma, gK\cap \supp (\chi)\neq \emptyset\}$ is  also finite.}
{We denote from now on for $k\geq 1$ and for any continuous cutoff function $\chi\in C(Z)$ for the proper $\Gamma$-action on $Z$ with $W_\chi:=\{\chi\neq 0\}$,} by $\Gamma_\chi^{(k)}$ the subset of $\Gamma^2$ given by
$$
{\Gamma_{\chi}^{(k)} := \{(g_0,g_k)\in \Gamma^2\,\vert\, \exists (g_i)_{1\leq i\leq k-1}\text{ such that }g_iW_\chi\cap g_{i+1}W_\chi\neq \emptyset \text{ for }0\leq i < k\}.}
$$
{For $k=0$, we simply take for $\Gamma_\chi^{(0)}$ the diagonal in $\Gamma^2$ which is isomorphic to $\Gamma$. {We point out that  the first and second projections $\Gamma_\chi^{(k)}\to \Gamma$ are proper, this is  indeed equivalent to the same statement for $k=1$ where the statement is clear from the above definitions.}

\begin{definition}\label{NewDefinition1.1}
Let $Z$ be a given proper $\Gamma$-space. We shall say that the (non-negative) cutoff function $\chi$ is uniform if {$\{g\in \Gamma, g\supp (\chi)\cap \supp (\chi)\neq \emptyset\}$ generates the group $\Gamma$.}
%:
%\begin{itemize}
%\item[(i)] The first (and/or second) projection $\Gamma_\chi^{(1)}\to \Gamma$ is proper;
%\item[(ii)] The subset $\{g\in \Gamma, g\supp (\chi)\cap \supp (\chi)\neq \emptyset\}$ generates the group $\Gamma$.
%\end{itemize}
\end{definition}

{In particular, the existence of such uniform cutoff function implies that $\Gamma$ is finitely generated.  Conversely, if $\Gamma$ is finitely generated, then every proper $\Gamma$-space  admits uniform cutoff functions.  More precisely:}
{\begin{lemma}
{Let $\Gamma$ be a countable group, then the following are equivalent:}
\begin{itemize}
\item[(i)] {$\Gamma$ is finitely generated;}
\item[(ii)] {There exists a proper  $\Gamma$-space with a uniform cutoff function;}
\item[(iii)] {Every proper $\Gamma$-space has a uniform cutoff function.}
\end{itemize}
\end{lemma}
\begin{proof}
{It is clear from Definition \ref{NewDefinition1.1} that (ii) implies (i). Since $\Gamma$ has proper actions (for instance $Z=\Gamma$ itself), it is also obvious that (iii) implies (ii). Assume now that $\Gamma$ is finitely generated and that $Z$ is a proper $\Gamma$-space. Choose some non-negative cutoff function $\chi_0$ on $Z$. }Let $F$ be a finite symmetric generating subset of $\Gamma$. Recall that this means that the finite subset $F$ generates the group $\Gamma$ and satisfies that $e\in S$ and for any $g\in F$,  $g^{-1}\in F$.
We then set
$$
\chi:=\frac{1}{\vert F\vert} \sum_{g\in F} g\chi_0.
$$
It is then clear that $\chi$ is a new non-negative cutoff function and that  $F\subset \{g\in \Gamma, g\supp (\chi)\cap \supp (\chi)\neq \emptyset\}$.
\end{proof}

%{In contrast, when $\Gamma$ is not finitely generated, and even for cocompact $Z$, the condition (ii) of Definition \ref{NewDefinition1.1} is necessary, as can be checked for $Z=\Gamma$ itself.}

{\begin{lemma}
Assume that $\Gamma$ is finitely generated and that $\chi$ is a uniform cutoff function on the proper $\Gamma$-space $Z$. Then
$$
\Gamma^2=\bigcup_{k\geq 0} \Gamma_\chi^{(k)}.
$$
\end{lemma}}

\begin{proof}
%We first point out that (i) of the previous definition is equivalent to the following property
%$$
% (iii) \;\forall k\geq 1, \text{ the first (and/or second) projection }\Gamma_\chi^{(k)}\to \Gamma\text{ is proper}.
%$$
%ADD AN ARGUMENT HERE.
{Set $A=\{g\in \Gamma, gV_\chi \cap V_\chi\neq \emptyset\}$.} Notice that
$$
(g,g')\in \Gamma_\chi^{(k)} \Longleftrightarrow \exists (g_i)_{1\leq i\leq k-1}\text{ such that }g_{i+1}^{-1} g_{i}\in A \text{ for }0\leq i\leq k-1, \text{ with }g_0=g\text{ and } g_k=g'.
$$
This shows that ${g'}^{-1}g\in A^k$, and this latter condition is clearly equivalent  to the condition $(g,g')\in \Gamma_\chi^{(k)}$. Indeed, given $(a_1,\cdots, a_k)\in A^k$ such that ${g'}^{-1}g=a_k\cdots a_1$, by setting $g_0=g$ and $g_{i+1}=g_i a_i^{-1}$ one deduces immediately that $(g,g')\in \Gamma_\chi^{(k)}$.

{Now, since we assume that $A$ generates $\Gamma$,  we can write any $(g,g')\in \Gamma^2$ as a product
$(k_1,k'_1)\cdots (k_r, k'_r)$ for some elements $(k_j)_{1\leq j\leq r}$ and $(k'_j)_{1\leq j\leq r}$ in $A$.
Using that $A$ is symmetric, we conclude that ${g'}^{-1}g \in A^{2r}.$}
\end{proof}

%\begin{remark}
%If $\Gamma_i$ acts uniformly properly on $Z_i$ for $i=1, 2$ then the obvious action of $\Gamma_1\times \Gamma_2$ on $Z_1\times Z_2$ will automatically be uniformly proper. Examples of proper non uniformly proper actions can be found in the literature, they are given by finitely generated infinite torsion groups, see for instance \cite{Osin:16}.
%\end{remark}

{{In the present paper we shall need to work with uniform cutoff functions and will therefore assume from now on that $\Gamma$ is finitely generated. Then, all cutoff functions will be chosen uniform.} Then a given operator $T$ will  have {\em{localized support}}  if there exists an  integer $k\geq 0$ such that the support of $T$ is contained in the closure of
$$
\bigcup_{(g,g')\in \Gamma_\chi^{(k)}} \; g W_\chi\times g'W_\chi.
$$
As explained in Appendix \ref{LocalizedOperators}, this is equivalent to the existence of $k'\geq 0$ such that the support of $T$ is contained in $\bigcup_{(g,g')\in \Gamma_\chi^{(k')}} \; gW_\chi\times g'W_\chi$.
The propagation index of $T$ is then the least such $k'$. }

{When the action of $\Gamma$ is cocompact and $Z$ is endowed with a $\Gamma$-invariant metric which endows it with the topology of a proper-metric space, it is easy to see that localized operators coincide with finite propagation operators, see again Appendix \ref{LocalizedOperators}. }

Let us recall now the notion of  fibrewise ample representation, see  \cite{PPV, BR2}.

\begin{definition}\label{FiberwiseAmple} A $C(X)$-representation $\what\pi: C (X, A)\rightarrow \maL_{C(X)} (E)$ will be called a fibrewise ample representation if for any $x\in X$,
the representation $\pi_{x}: A\rightarrow \maL (E_x)$ is ample, i.e.  for any $x\in X$, $\pi_x$ is non-degenerate and one has  for $a\in A$:
$$
\pi_{x} (a) \in \maK(E_x) \Longrightarrow a=0.
$$
Here and as usual $\maK(E_x)$ denotes the elementary $C^*$-algebra of compact operators on the Hilbert space $E_x$.
\end{definition}

Given a Hilbert space unitary representation $U: \Gamma\to U(H)$, we denote as usual by $H^\infty$ the Hilbert space $H\otimes\ell^2\N$ endowed with the unitary representation $U\otimes \id_{\ell^2\N}$.  Unless otherwise specified, the Hilbert space $\ell^2\Gamma$ will be endowed with the right regular representation of $\Gamma$, so $\ell^2\Gamma^\infty$ is endowed with the corresponding representation. Our extended PPV theorem can be stated as follows:\\

\medskip

\begin{theorem}\label{GequivVRiso}\
{Assume that the action of $\Gamma$ on $Z$ is  proper with a uniform cutoff function $\chi$.} Let   $H_1$ and $H_2$ be two infinite-dimensional separable complex Hilbert spaces,  endowed with unitary representations of $\Gamma$.  Let $\what\pi_1$ and $\what\pi_2$ be as above two fiberwise ample $\Gamma$-equivariant $C(X)$-representations of
$C(X, A)$ in the Hilbert $\Gamma$-modules $C(X)\otimes H_1$  and $C(X)\otimes H_2$ respectively.
Then, identifying  each $\what\pi_i$ with the trivially extended representation $\left(\begin{array}{cc} \what\pi_i & 0\\ 0 & 0\end{array}\right)$ that is further tensored by the identity
of $\ell^2\Gamma^\infty$ there exists a sequence $\{W_n\}_{n\in \N}$ of {\underline{$\Gamma$-invariant}}  unitary operators
$$
W_n\in \maL_{C(X)} \left([(H_1\oplus H_2)\otimes \ell^2\Gamma^\infty]\otimes C(X) ,  [(H_2\oplus H_1)\otimes \ell^2\Gamma^\infty] \otimes C(X)\right),
$$
such that
$$
W^*_n\what\pi_2 (\varphi) W_n - \what\pi_1(\varphi)\text{ is compact, and }  \lim_{n\rightarrow\infty} ||W^*_n\what\pi_2 (\varphi) W_n - \what\pi_1(\varphi)||= 0.
$$
  Moreover, {we can ensure that the operators $W_n$ are localized with uniform propagation index, actually $\leq 7$. In particular,} if $Z$ is proper-metric such that  $Z/\Gamma$ is compact then we can always ensure that the unitaries $W_n$ have (uniform) \underline{finite propagation}.
\end{theorem}

\medskip

In the next section, Theorem \ref{GequivVRiso} is first partially proved, more precisely we prove  the weaker version stated as Theorem \ref{GequivVRiso2},  which only constructs one unitary $W$ with the allowed properties. It is only later on in Subsection \ref{ControlNow} that the construction of the sequence $(W_n)_n$ is carried out with the norm-control.  In the sequel, an isometry (resp. unitary) $S$ satisfying the (up to compact operators) intertwining property \eqref{intertwine} will be referred to as a PPV-isometry (resp. PPV-unitary).\\

\medskip

\subsection{Equivariant Paschke-Higson duality}  As an important application, we deduce the Paschke-Higson duality isomorphism for $\Gamma$-families. More precisely, we assume now and for simplicity that the action of $\Gamma$ on $Z$ is {\underline{cocompact}} {and that $Z$ is a proper-metric space with a chosen $\Gamma$-invariant distance}.
%Notice though that the general case can be treated similarly using the generalized Roe algebras replacing finite propagation by localized operators {under the assumption of uniform properness of the $\Gamma$-action}, see Remark \ref{RoeLocalized}.
In \cite{BR2}, we defined  a generalization of the classical equivariant Roe algebras of pseudolocal and locally compact operators associated
with a fiberwise ample representation of $C(X, A)$ on the Hilbert $C(X)$-module $(\ell^2\Gamma^\infty\otimes H)\otimes C(X)$ induced by a given ample representation of $A$ in a fixed $H$. The Roe algebra of pseudolocal operators is
denoted by $D^*_\Gamma(X, A; (\ell^2\Gamma^\infty\otimes H))$, and the Roe algebra of locally compact operators is denoted $C^*_\Gamma(X, A; (\ell^2\Gamma^\infty\otimes H))$. More precisely,  $D^*_\Gamma(X, A; (\ell^2\Gamma^\infty\otimes H))$ is by definition the norm closure in $\maL_{C(X)} \left( C(X)\otimes (\ell^2\Gamma^\infty\otimes H)\right)$  of the space
of $\Gamma$-invariant operators with finite propagation and whose commutators with the elements of $C(X, A)$ are  compact operators. The $C^*$-algebra $C^*_\Gamma(X, A; (\ell^2\Gamma^\infty\otimes H))$ is on the other hand its subspace which is composed of those operators which satisfy  the additional condition that their composition with the elements of $C(X, A)$ are already compact operators.

An obvious observation is that $C^*_\Gamma(X, A; (\ell^2\Gamma^\infty\otimes H))$ is  a 2-sided closed ideal in the unital $C^*$-algebra $D^*_\Gamma(X, A; (\ell^2\Gamma^\infty\otimes H))$. Our Paschke-Higson
duality theorem identifies the $K$-theory of the quotient Roe algebra $Q^*_\Gamma(X, A;(Z,\ell^2\Gamma^\infty\otimes L^2Z))$ with the $\Gamma$-equivariant $KK$-theory of the pair of $\Gamma$-algebras $(A, C(X))$. More precisely:

\medskip

\begin{theorem}\label{Paschke2}
Suppose again that the isometric action of $\Gamma$ on $Z$ is proper and cocompact. Then we have a group isomorphism
$$
\maP_*: K_*(Q^*_\Gamma(X, A; (Z, \ell^2\Gamma^\infty\otimes H))) \xrightarrow{\cong} KK^\Gamma_{*+1}(A, C(X)), \quad \quad  *= 0, 1.
$$
\end{theorem}

Notice that $Z$ does not appear in the RHS, only its existence is supposed so that the LHS does not depend on the choice of such $Z$. The fact that the ample representation does not appear in the RHS is standard due to our PPV theorem.  In the case of trivial $\Gamma$, this theorem is well known, see for instance \cite{HigsonPaschke, Valette}. An interesting case corresponds to the case $A= C_0(Z)$. Then we get using the notations from \cite{BR2} the following theorem which was fully used there to deduce the Higson-Roe sequence for the groupoid $G=X\rtimes \Gamma$:

\begin{theorem}\label{Paschke3}
Suppose again that the isometric action of $\Gamma$ on $Z$ is proper and cocompact. Then we have a group isomorphism
$$
\maP_*: K_*(Q^*_\Gamma(X; (Z,\ell^2\Gamma^\infty\otimes H))) \xrightarrow{\cong} KK_\Gamma^{*+1}(Z, X), \quad \quad  *= 0, 1.
$$
\end{theorem}

\medskip

%\end{theorem}

%\begin{remark}
%When $X$ is compact, the assumption reduces to the properness and cocompactness of the $\Gamma$-action on $Z$, and the group $KK_G^{*+1}(X\times Z, X)$ is then isomorphic to the group $KK_\Gamma^{*+1}(Z, X)$. Therefore, we get in this case an isomorphism
%$$
%\maP_*: K_*(Q^*_\Gamma(X;(Z,\ell^2\Gamma^\infty\otimes L^2Z))) \xrightarrow{\cong} KK_\Gamma^{*+1}(Z, X), \quad \quad  *= 0, 1.
%$$
%which was fully used in \cite{BR2}.
%\end{remark}
%\medskip

The proof of our PPV theorem as well as the deduction of the Paschke-Higson isomorphism, are carried out in the next sections.

\medskip

\section{Proof of the extended PPV theorem}

{We devote this section to the proof of our $G$-equivariant, norm-controlled and support-localized, version of the PPV theorem, say Theorem \ref{GequivVRiso}. Inorder to simplify the reading of the this technical proof, we have first given the proof of a weaker version which does not adress the norm-control question.}

\subsection{The support-localized PPV theorem}
{We first forget the norm-control and prove the following weaker version of Theorem \ref{GequivVRiso}.}

\begin{theorem}[Extended PPV theorem]\label{GequivVRiso2}\
Assume that the action of $\Gamma$ on $Z$ is  proper with a uniform cutoff function $\chi$. Let   $H_1$ and $H_2$ be two infinite-dimensional separable complex Hilbert spaces,  endowed with unitary representations of $\Gamma$.  Let $\what\pi_1$ and $\what\pi_2$ be as above two fiberwise ample $\Gamma$-equivariant $C(X)$-representations of
$C(X, A)$ in the Hilbert $\Gamma$-modules $C(X)\otimes H_1$  and $C(X)\otimes H_2$ respectively.
Then, identifying  each $\what\pi_i$ with the trivially extended representation $\left(\begin{array}{cc} \what\pi_i & 0\\ 0 & 0\end{array}\right)$ that is further tensored by the identity
of $\ell^2\Gamma^\infty$, there exists a {\underline{$\Gamma$-invariant}}  unitary operator
$$
W\in \maL_{C(X)} \left([(H_1\oplus H_2)\otimes \ell^2\Gamma^\infty]\otimes C(X) ,  [(H_2\oplus H_1)\otimes \ell^2\Gamma^\infty] \otimes C(X)\right),
$$
which {essentially} intertwines the extended representations, i.e. such that
$$
  W^*\what\pi_2 (\varphi) W - \what\pi_1(\varphi) \in \maK_{C(X)} \left([(H_1\oplus H_2)\otimes \ell^2\Gamma^\infty]\otimes C(X)\right), \text{ for all } \varphi\in C(X, A).
  $$
  Moreover, {we can ensure that the operator $W$ is localized. In particular,} if the proper $\Gamma$-space $Z$ is cocompact then we can ensure that the unitary $W$ has \underline{finite propagation}.\end{theorem}

\medskip

{Under the assumption that $A=C_0(Z)$ for a proper and cocompact $\Gamma$-space $Z$, a striking application of Theorem \ref{GequivVRiso2} is to the equivariant family Paschke-Higson duality Theorem \ref{Paschke2}, as stated in Section \ref{PaschkeDuality} and which allows to incorporate  the Baum-Connes map for the groupoid $G=X\rtimes \Gamma$ in a long six-term exact sequence, see \cite{BR2}.} Notice that if $\Gamma$ is a finite  group then any separable $\Gamma$-algebra is a proper $\Gamma$-algebra over the trivial space $Z=\{\star\}$, and the theorem is valid for any such $\Gamma$-algebra. This is well known, see  \cite{PPV} for trivial $\Gamma$ and unital $A$, and  \cite{Kasparov1} for the general case of compact group actions.
Forgetting first the $\Gamma$-invariance of the intertwining unitary, we shall first prove the following independent result:

\medskip

\begin{theorem}\label{NCVRisometry2} Under the assumption and notations of Theorem \ref{GequivVRiso2} but for any proper $\Gamma$-action on $Z$,  there exists a unitary
$$
{U}\in \maL_{C(X)} \left( [\ell^2\Gamma^\infty\otimes (H_1\oplus H_2)]\otimes C(X), [\ell^2\Gamma^\infty\otimes (H_2\oplus H_1)] \otimes C(X)\right)
$$
such that
\begin{equation}\label{intertwine}
  {U^*\what\pi_2^\infty (\varphi) U}-\what\pi_1^\infty(\varphi) \in  \maK_{C(X)} \left([\ell^2\Gamma^\infty\otimes (H_1\oplus H_2)]\otimes C(X) \right) \text{ for all } \varphi\in C(X, A).
\end{equation}
  Moreover, we can ensure that the operator {the operator $U$ is localized. In particular,}  if the proper $\Gamma$-space $Z$ is cocompact, then we can ensure that $U$ has {\underline{finite propagation}}.
\end{theorem}

We thus fix  two fiberwise ample $G$-equivariant representations $\what\pi_1$ and $\what\pi_2$ of
$C(X, A)$ in the Hilbert $G$-modules $H_1\otimes C(X)$  and $H_2\otimes C(X)$ respectively.
In the sequel, we shall denote for $x\in X$, by  $q_x$ the composite map
$$
q_x: \, \maL_{C(X)} (H_i\otimes C(X))\xrightarrow{d_x} \maL(H_i)\xrightarrow{\pr} Q(H_i),
$$
where $d_x$ is evaluation at $x$ while the map $\maL(H_i)\xrightarrow{\pr} Q(H_i)$ is the quotient projection onto the Calkin algebra $Q(H_i)=\maL(H_i)/\maK (H_i)$. We begin with the construction of a finite-propagation PPV-\textit{isometry}.

%We also identify each $\what\pi_i$ with the trivially extended representation $\left(\begin{array}{cc} \what\pi_i & 0\\ 0 & 0\end{array}\right)$ that is further tensored by the identity
%of $\ell^2\Gamma^\infty$.

%\subsection{The compact case}

\begin{lemma}\label{finpropVRiso1}
 {Under the assumptions of Theorem \ref{NCVRisometry2},} there exists an isometry
 $$
 \hat{S} \in \maL_{C(X)}((H_1\oplus H_2)\otimes C(X),  (H_2\oplus H_1)\otimes \ell^2\Gamma \otimes C(X))
 $$
 such that
 $$
 \hat{S}^*((\what\pi_2(f)\otimes \id_{\ell^2\Gamma})\oplus 0)\hat{S}- (\what\pi_1(f)\oplus 0) \in \maK_{C(X)}((H_1\oplus H_2)\otimes C(X)), \quad  \forall f\in C_0(X\times Z)
 $$
 Moreover, we can ensure that the operator $\hat{S}$ {is localized}. In particular, when $Z/\Gamma$ is compact, we can ensure that $\hat{S}$ has finite propagation.
\end{lemma}

\begin{proof}
When $\Gamma$ is a finite group, this result is well known, see for instance \cite{Kasparov1}, and we give the proof under the assumption that $\Gamma$ is infinite. Fix a cutoff function $\chi\in C(Z)$ for the proper $\Gamma$-action on $Z$. The quotient projection $Z\to Z/\Gamma$ then restricts into a proper map $\Supp(\chi) \to Z/\Gamma$. Denote by $V_\chi$ the interior of $\Supp(\chi)$ and {let $C_0(Z, \chi)$ be the $C^*$-subalgebra of $C_b(Z)$ generated by $C_0(Z)$ and $\chi$, and in the same way let $C_0(\chi)$ be the $C^*$-subalgebra of $C_b(Z)$ generated by $C_0(V_\chi)$ and  $\chi$. Elements of $C_0(Z, \chi)$ (resp. of $C_0(\chi)$) are then (uniform limits of) functions of type $f+\varphi \circ \chi$ where $f\in C_0(Z)$ (resp.  $f\in C_0(V_\chi)$) and $\varphi$ is any continuous function on $\R_+$ with $\varphi (0)=0$.  It is an obvious observation that $C_0(Z, \chi)\subset C_{0\vert\Gamma} (Z)$ and that  $C_0(\chi)$ is an ideal in $C_0(Z, \chi)$. Notice that the $C^*$-algebra $C_0(Z, \chi)$ as well as its ideal $C_0(\chi)$ are separable.
Recall that the $*$-homomorphism $\pi_i: A\to \maL_{C(X)} (H_i\otimes C(X))$ is non-degenerate and extends to a unital $*$-homomorphism, still denoted $\pi_i$, from $M(A)$ to $\maL_{C(X)} (H_i\otimes C(X))$.  We shall  identify $C_0(Z, \chi)$ with its range in the center of the multiplier algebra {of} $A$ when no confusion can occur.}

{We now
%If for instance $Z/\Gamma$ is compact, then $V_\chi$ is relatively compact and $\chi\in C_c(Z)$ but this is not true in general.
set  $A_\chi:= C_0(\chi) A^+$, then $A_\chi$ is a two-sided self-adjoint ideal in $C_0(Z, \chi) A^+$.}

%, independent of the $X$-variable, since $X$ is assumed to be compact here.Thus
%$c$ is constant in the $X$-variable and $\Supp_{X\times Z}(c)= \pr_1^*(\Supp_Z(c|_Z))= X\times \Supp_Z(c|_Z)$. Note that $\Supp_Z(c|_Z)$ can be chosen to have non-empty interior.
 The proof of Lemma \ref{finpropVRiso1} will be split into 3 steps.

\textbf{Step 1{\em {(Apply PPV)}}:} Consider the Hilbert submodules $\maE^i_\chi:= \overline{\pi_i(A_\chi)(H_i\otimes C(X))}$ for $i=1,2$.

\textbf{Claim}: {{The $*$-homomorphism $\pi_i$, {extended to $C_0(Z, \chi) A^+$}, preserves $\maE^i_\chi$ {since the image of $C_0(Z, \chi)$ in $M(A)$ is valued in the center $ZM(A)$,} and its restriction  to $A_\chi$ (acting on $\maE^i_\chi$) is denoted $\pi_i^\chi$. This is again associated with a fibrewise ample representation $\what\pi_i^\chi$.}}

% , when amplified by tensoring with identity on $\ell^2\N$,
% are then fibrewise ample representations of $C_0(U_\chi)$ on $\ell^2\N\otimes \maE^i_\chi$, i.e. $\id_{\ell^2\N}\otimes \what\pi_i^\chi$ is a fibrewise ample representation.

To check this, note that the field of Hilbert spaces associated to the Hilbert module $\maE^i_\chi$ is given by $(R_{{i},x}:=\overline{[\pi_{i,x}(A_\chi)H_i]})_{x\in X}$,
where $(\pi_{i,x})_{x\in X}$ is the field of
representations associated with the $*$-homomorphism $\pi_i: M(A)A\to \maL_{C(X)} (H_i\otimes C(X))$.
%Indeed,  for any $x\in X$, the map $\kappa_x: \maE^i_\chi\otimes_{\ev_x}\C\rightarrow R_{{i},x}$ given, for $a \in A_\chi$ and $\xi\otimes h\in H_i\otimes C(X)$, by:
%$$
%\kappa_x[\pi_i^\chi(a)(\xi\otimes h)\otimes_{\ev_x} 1]:= \pi_i^\chi(a)(\xi\otimes h)(x)= h(x)[\pi_{i,x}(a)\xi],
%$$
% is an isometric isomorphism.
%
Then, the field of Hilbert space representations associated with  $\pi^\chi_i$ is given at $x\in X$ by the restriction of the representations $\pi_{i,x}$ and denoted
$\pi_{i, x}^\chi: A_\chi\rightarrow \maL(R_{{i},x})$. For each $x\in X$, the restricted representations $\pi_{i,x}^\chi$ are also {clearly} ample.

{The $C(X)$-module $\maE^i_\chi$ is countably generated and we may assume that it is   an orthocomplemented submodule of $H_i\otimes C(X)$, respectively for $i=1,2$.}
The $*$-homomorphism $\pi_i^\chi$ will be extended by zero on the orthocomplement, {this corresponds to extending each $\pi_{i, x }^\chi$ by zero on the orthogonal Hilbert subspace of $R_{i, x}$}. The corresponding extended representation  $\what\pi_i^\chi\oplus 0$ of $C(X, A_\chi)$ then satisfies the following properties \cite{PPV}:
\begin{enumerate}
\item  \underline{it is lower semi-continuous}, i.e.  for any convergent sequence $x_n\rightarrow x_0$ in $X$, we have $\bigcap_{n} I_{x_n} \subseteq I_{x_0}$ where
$I_x:= \ker(q_x\circ(\what\pi_i^\chi\oplus 0))\subseteq A_\chi$;
\item  \underline{it is exact}, i.e. $\cap_{x\in X} I_x =\{0\}$, and
 \item \underline{it is trivial}, i.e.  $\ker(q_x\circ (\what\pi_i^\chi\oplus 0))= \ker(d_x\circ (\what\pi^\chi_i\oplus 0))$.
\end{enumerate}

First note that (3) is automatically satisfied since $\what\pi^i_\chi$ is fibrewise ample.
%Indeed,  if $f\in \ker(q_x\circ (\what\pi_i^\chi\oplus 0))$, then
%$\pi^\chi_{i,x}(f_x)\in \maK(R_x)\implies f_x= 0$, so $f\in \ker(d_x\circ (\what\pi_i^\chi\oplus 0))$.
To check (1), just notice that
$$
d_{x}\circ (\pi_i^\chi(a)\oplus 0)= (\pi_{i, x}^\chi\oplus 0)(a),
$$
and that each $\pi_{i,x}$ is ample here. A similar but easier argument can be used to prove (2).
%, indeed if $f(x, V_\chi)=0$ for all $x\in X$ then $f$ is identically zero on $U_\chi$.

\medskip

The $C(X)$-representation $\what\pi_i^\chi\oplus 0$ is given by the $*$-homomorphism $\pi_i^\chi\oplus 0 : A_\chi \to \maL_{C(X)} (H_i\otimes C(X))$. In order to apply the main PPV theorem about {\em{trivial $X$-extensions}}, we check now the same properties for the unique  extension of $\pi_i^\chi\oplus 0$ to the { $C^*$-algebra unitalization $A_\chi\oplus \C$ of $A_\chi$. In terms of $C(X)$-representations, we thus obtain the extended representation to $C(X, A_\chi)\oplus C(X)$  given by}
$$
(\what\pi_i^\chi\oplus 0)^+(f \oplus \lambda):= (\what\pi_i^\chi\oplus 0)(f) + \rho (\lambda)\text{ for } f\in C(X, A_\chi)\text{ and }\lambda\in C(X).
% \begin{bmatrix}   \what\pi^\chi_2(f)+\lambda \id_{\maE^i_\chi} &0\\ 0 & \lambda \id_{(H^\chi_i)^\perp}   \end{bmatrix}
$$
{Here $\lambda$ acts on the Hilbert $C(X)$-module $H_i\otimes C(X)$ by the adjointable operator $\rho (\lambda)$ corresponding to the right module multiplication.}

{Since $\Gamma$ is infinite all the Hilbert spaces $R_{i, x}^\perp$ are infinite dimensional separable Hilbert spaces}, we may use  the  Kasparov stabilisation theorem to replace
$({\maE_\chi^i})^\perp$ by  the standard infinite dimensional countably generated  Hilbert $C(X)$-module  $H_i\otimes C(X)$, {so as to be able to apply the PPV theorem{, see \cite{Kasparov1}}.}

Hence the verification of the three properties for  $(\what\pi_i^\chi\oplus 0)^+$ is obvious.
If for instance $(f,\lambda)\in \ker(q_x\circ (\what\pi_i^\chi\oplus 0)^+)$, then $\lambda =0$ since $({\maE_\chi^i})^\perp_x$ is infinite-dimensional. Thus,
we again get $\pi^\chi_{i,x}(f (x))\in \maK(R_{{i}, x})\implies f(x)= 0$, so $(f,\lambda)\in \ker(d_x\circ (\what\pi_i^\chi\oplus 0)^+)$. Therefore, we get the triviality property.
Lower-semicontinuity and exactness are proved  similarly and are left as an exercise.
%, we define the analogous ideals $I_x^+:= \ker(q_x\circ(\what\pi_i^\chi\oplus 0)^+)$. We need to check that if $x_n\rightarrow x_0$, then
%$\bigcap_{n} I_{x_n}^+ \subseteq I_{x_0}^+$. Let $(f,\lambda)\in \bigcap_{n} I_{x_n}^+$. As before, we deduce that $\lambda=0$, and thus $f\in \bigcap_{n} I_{x_n}$. However, we have already shown
%that this implies $f\in I_{x_0}$. Thus $(f,\lambda)= (f,0)\in I_{x_0}^+$.

%The exactness property of the $X$-extension is proved similarly.
{Therefore, the representations $(\what\pi^\chi_i\oplus 0)^+$ are essentially unitarily equivalent by the PPV theorem {\cite{PPV}[Theorem 2.10], i.e.}} there exists  a unitary
$S_\chi\in \maL_{C(X)} (H_1\otimes C(X), H_2\otimes C(X))$ such that we have in particular for any $f\in C(X, A_\chi)$:
$$
S^*_\chi\begin{bmatrix}   \what\pi^\chi_2(f) &0\\ 0 & 0   \end{bmatrix} S_\chi -\begin{bmatrix}   \what\pi^\chi_1(f) &0\\ 0 & 0   \end{bmatrix} \in \maK_{C(X)} (H_1\otimes C(X)).
$$
Notice that  we then have the same relation for ${f\in C(X, C_0(\chi))}$ viewed in $C(X, A_\chi)$, {hence in particular for $f\in C_0(X\times V_\chi)$ but also for $\chi$ as well as any continuous function of $\chi$ vanishing at zero, for instance for $\chi^{1/2}$}.

\textbf{Step 2 ({\em{First modification}}):}  Consider the operator $s_\chi \in \maL(\maE^1_\chi, \maE^2_\chi)$ which is the $(1,1)$-entry in the matrix decomposition of
$S_\chi: \maE^1_\chi \oplus (\maE^1_\chi)^\perp\rightarrow \maE^2_\chi\oplus (\maE^2_\chi)^\perp$.
It satisfies the following properties for any $f\in C(X, A_\chi)$:
\begin{enumerate}
\item $s^*_\chi\what\pi^\chi_2(f)s_\chi- \what\pi^\chi_1(f) \sim 0$, %\quad S_{12}^*\what\pi^\chi_{Y}^\infty(f) \sim 0, \quad \what\pi^\chi_Y^\infty(f)S_{12} \sim 0$
\item $\what\pi^\chi_2(f)(s_\chi s^*_\chi- \id) \sim 0$ and  $(s^*_\chi s_\chi- \id)\what\pi^\chi_1(f)\sim 0$,
%\item $\what\pi^\chi_Y^\infty(f)(ss^*s- s) \sim 0, \quad (s^*ss^*- s^*)\what\pi^\chi_Y^\infty(f)\sim 0$
\item $[s_\chi s^*_\chi,\what\pi^\chi_2(f)]\sim 0$ {{and $[s^*_\chi s_\chi,\what\pi^\chi_1(f)]\sim 0$}},
%\item $[s(F\oplus \id)s^*, \what\pi^\chi_Y^\infty(f)]\sim 0$, and
%\item $\what\pi^\infty_Y(f)((s(F\oplus \id)s^*)^2- ss^*)\sim 0$
\item $(1-s^*_\chi s_\chi)$ and $(1-s_\chi s^*_\chi)$ are positive operators.
\end{enumerate}
{Recall that the notation  $A\sim B$ means that the difference $A-B$ is compact.} Therefore we can form the unitary $\hat{s}_\chi: \maE^1_\chi\oplus \maE^2_\chi\rightarrow \maE^2_\chi\oplus \maE^1_\chi$ given by the matrix
$$
\hat{s}_\chi:= \begin{bmatrix} s_\chi & (1-s_\chi s^*_\chi)^{1/2} \\ -(1-s^*_\chi s_\chi)^{1/2} & s^*_\chi \end{bmatrix}
$$
Properties $(1), (2)$ and $(3)$ above, imply that $\hat{s}_\chi$ intertwines the representations $\what\pi^\chi_1\oplus \what\pi^\chi_2$ and $\what\pi^\chi_2\oplus \what\pi^\chi_1$ up to compacts on $\maE^1_\chi\oplus \maE^2_\chi$ and $\maE^2_\chi\oplus \maE^1_\chi$,
respectively. Extending the unitary $\hat{s}_\chi$ by zero, we  get a partial isometry in $\maL_{C(X)}((H_1\oplus H_2)\otimes C(X), (H_2\oplus H_1)\otimes C(X))$, that we still denote by $\hat{s}_\chi$, {given by:
$$
\left(\begin{array}{cccc}
s_\chi & 0 & \sqrt{1- s_\chi s_\chi^*} & 0 \\
0 & 0 & 0 & 0\\
-\sqrt{1-s_\chi^*s_\chi} & 0 & s_\chi^* & 0 \\
0 & 0 & 0 & 0
\end{array}\right)
$$}
{{where we have written
$$
(H_1\oplus H_{2})\otimes C(X) = [\maE^1_\chi\oplus (\maE^1_\chi)^\perp ]\oplus [\maE^{2}_\chi\oplus (\maE^{2}_\chi)^\perp], \text{ and similarly for } (H_2\oplus H_{1})\otimes C(X).
$$
}}

\textbf{Step 3 ({\em{Second modification)}}:} For  $(x,g)\in G$, we denote by $V^i_{(x,g)}$  the unitary implementing the $G$-action on $H_i\otimes C(X)$. Then we define an operator $\hat{S} \in \maL_{C(X)}((H_1\oplus H_2)\otimes C(X), \ell^2\Gamma\otimes (H_2\oplus H_1)\otimes C(X))$ {{by setting  the following pointwise formula:
$$
\left[\hat{S}_x\right]_g:= \left(V^2_{(x,g)} \oplus V^1_{(x,g)}\right) \hat{s}_\chi \left(\what\pi_1(\chi^{1/2})\oplus \what\pi_2(\chi^{1/2})\right)
\left(V^1_{(x,g)} \oplus V^2_{(x,g)}\right)^{-1}.
$$}}
%\begin{lemma}\label{finitepropVRiso}
The operator $\hat{S}$ then satisfies the allowed properties in the statement Lemma \ref{finpropVRiso1}, as we prove it in Lemma \ref{finpropVRiso3} below. Therefore, the proof of Lemma \ref{finpropVRiso1} is now complete.
% that $\hat{S}$ satisfies the following properties:
%
%\begin{enumerate}
% \item $\hat{S}$ is an isometry.
% \item $\hat{S}$ intertwines $\what\pi_1\oplus 0$ and $(\id_{\ell^2\Gamma}\otimes \what\pi_2)\oplus 0$ up to compacts.
% \item $\hat{S}$ has propagation bounded above by $\diam(U_\chi)$.
%\end{enumerate}
%
%Thus we get the assertion in the statement of the lemma.
\end{proof}

{Properties of finite propagation operators as well as the notion of propagation index, are expanded in Appendix \ref{LocalizedOperators}, see in particular Definition \ref{LocalizedOp}.}

\begin{lemma}\label{finpropVRiso3}
The operator $\hat{S}$ satisfies the following properties:
\begin{enumerate}
 \item $\hat{S}$ is an isometry.
 \item $\hat{S}$ intertwines $\what\pi_1\oplus 0$ and $(\id_{\ell^2\Gamma}\otimes\what\pi_2)\oplus 0$ up to compacts.
 \item {The support of ${\hat S}$ is contained in the closure of $\maA_0=\bigcup_{g\in \Gamma} gW_\chi\times gW_\chi$ and hence has propagation index $\leq 2$. In particular,} if $Z/\Gamma$ is compact then $\hat{S}$ has finite propagation (bounded above by $\diam _Z(V_\chi)$).
\end{enumerate}
\end{lemma}
\begin{proof}

% \textbf{Step 3: Second modification:} Define an operator $\hat{S} \in \maL_{C(X)}((H\oplus H)\otimes C(X), \ell^2\N\otimes (H\oplus H)\otimes C(X))$ by the following pointwise formula: for
% $x\in X$, and $\xi=(\xi^{(1)},\xi^{(2)})$ a compactly supported element of $(H\oplus H)\otimes C(X)$,
%
% $$
% \hat{S}_x(\xi^{(1)}_x\oplus \xi^{(2)}_x):= ((V_{(x,g)} \oplus V_{(x,g)})\hat{s}_\beta (\what\pi_1^\beta(\phi_\beta^{1/2})\oplus \what\pi_2^\beta(\phi_\beta^{1/2}))
% (V^{-1}_{(x,g)} \oplus V^{-1}_{(x,g)})(\xi^{(1)}_x\oplus \xi^{(2)}_x))_{(\beta, g)\in B\times \Gamma}
% $$
% \begin{lemma}\label{finitepropVRiso}
% We claim that $\hat{S}$ satisfies the following properties:
%
% \begin{enumerate}
%  \item $\hat{S}$ is an isometry.
%  \item $\hat{S}$ intertwines $\what\pi_1\oplus 0$ and $\what\pi^\infty_2\oplus 0$ up to compacts.
%  \item $\hat{S}$ has propagation bounded above by $\sup_{\beta \in B} \diam(U_\beta)$.
% \end{enumerate}
% \end{lemma}
% \begin{proof}

(1) {Since $\hat{s}_\chi$ is an isometry in restriction to the range of $\what\pi_1(\chi^{1/2})\oplus \what\pi_2(\chi^{1/2}))$, a straightforward verification using  the relation $\sum_{ g\in \Gamma} g^*\chi=1_Z$ shows that ${\hat S}$ is an isometry.}

(2)
%$\hat{S}$ intertwines $\what\pi_1\oplus 0$ and $\id_{\ell^2\Gamma}\otimes\what\pi_2\oplus 0$ up to compacts:
It suffices to check this condition for elements  $f$ in $C(X, A_c)$ where $A_c:= C_c(Z)A\subset A$. {Then, using the previously listed properties of $\hat{s}_\chi$, we have}:
\begin{eqnarray*}
&&{[\hat{S}^*\left(\id_{\ell^2\Gamma}\otimes \what\pi_2(f)\oplus 0\right)\hat{S}]_x}\\
%&=& \hat{S}^* \left(\what\pi_2(f)\oplus 0\right)\left(V^2_{(x,g)} \oplus V^1_{(x,g)}\right)\hat{s}_\chi\left(\what\pi_1(\chi^{1/2})\oplus \what\pi_2(\chi^{1/2})\right)
%\left((V^1)^{-1}_{(x,g)} \oplus (V^2)^{-1}_{(x,g)}\right)  \\
&=& \sum_{g\in \Gamma} (V^1_{(x,g)} \oplus V^2_{(x,g)}) (\pi_1(\chi^{1/2})\oplus \pi_2(\chi^{1/2})) \hat{s}_\chi^* ((V^2)^{-1}_{(x,g)}\oplus (V^1)^{-1}_{(x,g)})(\what\pi_2(f)\oplus 0)\\
&&  (V^2_{(x,g)} \oplus V^1_{(x,g)})\hat{s}_\chi (\pi_1(\chi^{1/2})\oplus \pi_2(\chi^{1/2}))
((V^1)^{-1}_{(x,g)} \oplus (V^2)^{-1}_{(x,g)})\\
&=& \sum_{g\in \Gamma} (V^1_{(x,g)} \oplus V^2_{(x,g)}) (\pi_1(\chi^{1/2})\oplus \pi_2(\chi^{1/2})) (\hat{s}_\chi^* (\what\pi_2(g^*f)\oplus 0)\hat{s}_\chi)\\
&& (\pi_1(\chi^{1/2})\oplus \pi_2(\chi^{1/2})) ((V^1)^{-1}_{(x,g)} \oplus (V^2)^{-1}_{(x,g)})\\
&\sim& \sum_{g\in \Gamma} (V^1_{(x,g)} \oplus V^2_{(x,g)}) (\pi_1(\chi^{1/2})\oplus \pi_2(\chi^{1/2})) (\hat{s}_\chi^* (\what\pi_2(\chi^{1/2} g^*f)\oplus 0)\hat{s}_\chi  ((V^1)^{-1}_{(x,g)} \oplus (V^2)^{-1}_{(x,g)})
\end{eqnarray*}
The last equivalence is a consequence of the fact that $\hat{s}_\chi$ commutes up to compacts with $\pi_1(\chi^{1/2})\oplus \pi_2(\chi^{1/2})$. {Recall that $\chi^{1/2}$ belongs to $C_0(\chi)$ and hence to $A_\chi$. Therefore} we deduce
\begin{eqnarray*}
&&{[\hat{S}^*\left(\id_{\ell^2\Gamma}\otimes \what\pi_2(f)\oplus 0\right)\hat{S}]_x}\\
&\sim& \sum_{g\in \Gamma} (V^1_{(x,g)} \oplus V^2_{(x,g)}) (\pi_1(\chi^{1/2})\oplus \pi_2(\chi^{1/2}))  (\what\pi_1(\chi^{1/2} g^*f)\oplus 0)((V^1)^{-1}_{(x,g)} \oplus (V^2)^{-1}_{(x,g)})\\
&\sim& \sum_{g\in \Gamma} (V^1_{(x,g)} \oplus V^2_{(x,g)}) (\pi_1(\chi^{1/2})\oplus \pi_2(\chi^{1/2})) (\what\pi_1(g^*f)\oplus 0)(\pi_1(\chi^{1/2})\oplus \pi_2(\chi^{1/2})) ((V^1)^{-1}_{(x,g)} \oplus (V^2)^{-1}_{(x,g)})\\
&=& (\what\pi_1(f)\oplus 0) \sum_{g\in \Gamma} (V^1_{(x,g)} \oplus V^2_{(x,g)}) (\pi_1(\chi)\oplus \pi_2(\chi)) ((V^1)^{-1}_{(x,g)} \oplus )V^2)^{-1}_{(x,g)})\\
&=& (\what\pi_1(f)\oplus 0)
\end{eqnarray*}
The above computation is legal because the number of elements $g\in \Gamma$ such that
$$
\Supp(g^*f)\cap ({X\times }{\Supp(\chi)}) \neq \emptyset
$$
is finite, due to the properness of the $\Gamma$-action and the fact that $\chi$ is a cut-off function. {Indeed, we know from the very definition of $\chi$ that for any compact subspace $K$ of $Z$, the subset $\{g\in \Gamma\vert \Supp (\chi)\cap g K\neq \emptyset\}$ is finite, see for instance \cite{TuHyper}.}

(3)  Assume now that ${{W_1}}$ and ${{W_2}}$ are two open subspaces of $Z$ such that $W_1\times W_2$ does not intersect any subspace of $Z^2$ of the form $gW_\chi \times gW_\chi$, where $g$ runs over $\Gamma$, then for $a_i\in C_0(W_i)A$, we can compute
{\begin{eqnarray*}
\hspace{-0cm}&&\left[(\id_{\ell^2\Gamma}\otimes\pi_2(a_2)\oplus 0)\hat{S}(\pi_1(a_1)\oplus 0)\right]_g \\
&=&\left(\pi_2(a_2)\oplus 0\right)\left(V^2_{(x,g)} \oplus V^1_{(x,g)}\right)\hat{s}_\chi \left(\pi_1(\chi^{1/2})\oplus \pi_2(\chi^{1/2})\right)
\left(V^1_{(x,g)} \oplus V^2_{(x,g)}\right)^{-1} \left(\pi_1(a_1)\oplus 0\right)\\
&=&\left(V^2_{(x,g)} \oplus V^1_{(x,g)}\right)\left(\pi_2(g^{-1}a_2)\oplus 0\right)\hat{s}_\chi \left(\pi_1(\chi^{1/2})\oplus \pi_2(\chi^{1/2})\right)
\left(\pi_1(g^{-1}a_1)\oplus 0\right)\left(V^1_{(x,g)} \oplus V^2_{(x,g)}\right)^{-1}\\
&=&\left(V^2_{(x,g)} \oplus V^1_{(x,g)}\right)\left(\pi_2(g^{-1}a_2)\oplus 0\right)\hat{s}_\chi \left(\pi_1(\chi^{1/2} g^{-1}a_1)\oplus 0\right)
\left(V^1_{(x,g)} \oplus V^2_{(x,g)}\right)^{-1}
\end{eqnarray*}}
Therefore, we see that if for a given $g\in \Gamma$, we have $\chi^{1/2} g^{-1}a_1$ is non-zero then $\Supp (a_1)\cap g\Supp (\chi) \neq \emptyset$. But then by  hypothesis we know that since $W_1$ and $W_2$ are open we also have
$$
W_1\times W_2 \bigcap g\Supp (\chi) \times g\Supp(\chi) = \emptyset,
$$
and hence  necessarily $\Supp (a_2)\cap g\Supp (\chi) = \emptyset$, say that $\Supp (g^{-1}a_2) \cap \Supp (\chi)=\emptyset$. This in turn implies that
$$
\left(\pi_2(g^{-1}a_2)\oplus 0\right)\hat{s}_\chi = 0 \text{ since  the range of }\hat{s}_\chi\text{ is contained in }\maE^2_\chi\oplus \maE^1_\chi.
$$
Therefore we conclude that the operator
$(\id_{\ell^2\Gamma}\otimes\pi_2(a_2)\oplus 0)\hat{S}(\pi_1(a_1)\oplus 0)$ is trivial.

If we assume  that $Z/\Gamma$ is compact and that $Z$ is a metric-proper space with the above properties, then setting $\kappa:= \diam_{{Z}}(\Supp(\chi))$ which is now a finite positive number,
we can deduce by the same calculation  that whenever $a_1, a_2\in C_c(Z)A$ are such that $d (\Supp(a_1), \Supp(a_2))> \kappa$, one has by the $\Gamma$-invariance of the distance $d$ the same relation
$$
(\id_{\ell^2\Gamma}\otimes\pi_2(a_2)\oplus 0)\hat{S}(\pi_1(a_1)\oplus 0) = 0.
$$

%
%
%
%PREVIOUS PROOF:\\
%
%Now suppose that $\Supp(g^{-1}a_1)\cap \supp(\chi)\neq \emptyset$ for some $g\in \Gamma$. Then, the support of the factor $
%\hat{s}_\chi(\pi_1(\chi^{1/2}g^{-1}a_1)\oplus 0))$
%is within $\supp(\chi)\times \supp(\chi)$. On the other hand we have $d(\supp(g^{-1}a_2), \supp(g^{-1}a_1))>\kappa$, so $\supp(g^{-1}a_2)$ is disjoint from $\supp(\chi)$. Thus the factor $
%(\pi_2(g^{-1}a_{{2}})\oplus 0)\hat{s}_\chi(\what\pi_1(\chi^{1/2}g^{-1}a_1)\oplus 0))$ vanishes.
%So each term in the {field} representation of $(\id_{\ell^2\Gamma}\otimes \pi_2(a_2)\oplus 0)\hat{S}(\pi_1(a_1)\oplus 0)$ vanishes, and we thus finally get (as allowed)
%$$
%\left(\id_{\ell^2\Gamma}\otimes\pi_2(a_2)\oplus 0\right)\, \hat{S}\, \left(\pi_1(a_1)\oplus 0\right)\,=\,0 .
%$$
\end{proof}

\begin{corollary}\label{balance}
There exists an isometry
$S\in \maL_{C(X)}(\ell^2\Gamma^\infty\otimes (H_1\oplus H_2)\otimes C(X), \ell^2\Gamma^\infty\otimes (H_2\oplus H_1)\otimes C(X))$  such that
$$
S^*(\what\pi_2^\infty(f)\oplus 0)S -(\what\pi_1^\infty(f)\oplus 0) \in \maK_{C(X)} \left([\ell^2\Gamma^\infty\otimes (H_1\oplus H_2)]\otimes C(X)\right).
$$
Moreover, {we can ensure that the operator $S$ is localized with support contained  in the closure of $\maA_0=\bigcup_{g\in \Gamma} gW_\chi\times g W_\chi$ and hence with propagation index $\leq 2$. In particular,} if $Z/\Gamma$ is compact then $\hat{S}$ has finite propagation (bounded above by $\diam(V_\chi)$).
\end{corollary}

\begin{proof}
From Lemma \ref{finpropVRiso1}, we deduce an isometry
$\hat{S} \in  \maL_{C(X)}(\ell^2\Gamma^\infty\otimes(H_1\oplus H_2)\otimes C(X),\ell^2\Gamma \otimes\ell^2\Gamma^\infty\otimes (H_2\oplus H_1)\otimes C(X))$,
such that
$\hat{S}$ intertwines the representations $\what\pi_1^\infty\oplus 0$ and $(\id_{\ell^2\Gamma}\otimes \what\pi^\infty_2)\oplus 0$ up to compacts.
Consider a unitary $u_\infty:\ell^2\Gamma\otimes \ell^2\Gamma^\infty \rightarrow \ell^2\Gamma^\infty$.
The isometry $S_1:= (u_\infty\otimes \id_{(H_2\oplus H_1)\otimes C(X)})\circ \hat{S}$ then intertwines $\what\pi_1^\infty\oplus 0$ and $\what\pi^\infty_2\oplus 0$, up to compacts and still has
the same support as ${\hat S}$. In particular, it has uniform finite propagation when $Z/\Gamma$ is compact.
\end{proof}

We are now ready to prove Theorem \ref{NCVRisometry2}.

\begin{proof}(of Theorem \ref{NCVRisometry2})
Replacing, in the statement of Corollary \ref{balance}, $H_i$ by $\ell^2\N\otimes H_i$ and $\what\pi_i$ by $\what\pi_i^\infty = \id_{\ell^2\N}\otimes \what\pi_i$ for $i=1,2$, we obtain an {isometry with the prescribed support condition  (finite-propagation when $Z/\Gamma$ is compact and $Z$ is metric-proper)}
$$
S_0 \in \maL_{C(X)}\left((\ell^2\Gamma^\infty)^\infty\otimes (H_1\oplus H_2)\otimes C(X), (\ell^2\Gamma^\infty)^\infty\otimes (H_2\oplus  H_1)\otimes C(X)\right)
$$
such that for any $f\in C(X, A)$:
$$
S_0^*( (\what\pi_2^\infty)^\infty(f)\oplus 0)S_0 -(\what\pi_1^\infty)^\infty (f)\oplus 0) \in \maK_{C(X)}\left( (\ell^2\Gamma^\infty)^\infty \otimes (H_1\oplus H_2)\otimes C(X)\right).
$$
{The support of $S_0$ is more precisely contained in $\bigcup_{g\in \Gamma} \Supp (g\chi)\times \Supp (g\chi)$ and hence $S_0$ has finite propagation in the metric and cocompact case. Indeed, in this case and since the distance is $\Gamma$-invariant, the propagation is $\leq$ the diameter of $\Supp (\chi)$.} Let $r_\infty: \ell^2\N\otimes \ell^2\N\rightarrow \ell^2\N$ be a unitary. Composing $S_0$ with $r_\infty\otimes \id_{\ell^2\Gamma\otimes (H_1\oplus H_2)\otimes C(X)}$, we get an isometry
$$
S_1:= \left(r_\infty\otimes \id_{\ell^2\Gamma\otimes (H_1\oplus H_2)\otimes C(X)}\right)\circ S_0 \in \maL_{C(X)}\left((\ell^2\Gamma^\infty)^\infty\otimes (H_1\oplus H_2)\otimes C(X), \ell^2\Gamma^\infty\otimes (H_2\oplus  H_1)\otimes C(X)\right)
$$ which  satisfies
$$
S_1^*\left(\what\pi_2^\infty(f)\oplus 0\right)S_1 -\left((\what\pi_1^\infty)^\infty(f)\oplus 0\right) \in \maK_{C(X)}\left((\ell^2\Gamma^\infty)^\infty\otimes (H_1\oplus H_2)\otimes C(X)\right)
$$
and has the same support.

Consider the  operator
$
R_1:\ell^2\N\otimes \ell^2\Gamma^\infty\otimes (H_1\oplus H_2)\rightarrow \ell^2\N\otimes \ell^2\Gamma^\infty\otimes (H_1\oplus H_2)
$
defined by the following formula:
$$
R_1(h_1\oplus h_2\oplus\cdots)= 0\oplus h_1\oplus h_2\oplus\cdots
$$
Then $R_1$  induces a $C(X)$-linear isometry on $\ell^2\N\otimes \ell^2\Gamma^\infty\otimes (H_1\oplus H_2)\otimes C(X)$.
Consider also the operator $R_2: \ell^2\N\otimes \ell^2\Gamma^\infty\otimes (H_1\oplus H_2)\rightarrow \ell^2\Gamma^\infty\otimes (H_1\oplus H_2)$
defined by the formula:
$$
R_2(h_1\oplus h_2\oplus\cdots)= h_1
$$
Then $R_2$ induces a $C(X)$-adjointable co-isometry
$$
R_2\in \maL_{C(X)}\left((\ell^2\Gamma^\infty)^\infty \otimes (H_1\oplus H_2) \otimes C(X), \ell^2\Gamma^\infty\otimes (H_1\oplus H_2)\otimes C(X)\right).
$$
Notice that we have the convenient relations
$$
{R_2R_1=0\text{ and }}R_1R_1^*+ R_2^*R_2= \id_{\ell^2\N\otimes\ell^2\Gamma^\infty\otimes (H_1\oplus H_2) \otimes C(X)}.
$$
We are  now in position to define the unitary
$$
S \in \maL_{C(X)}\left(\ell^2\Gamma^\infty\otimes (H_2\oplus H_1)\otimes C(X), \ell^2\Gamma^\infty\otimes \left((H_2\oplus H_1)\oplus (H_1\oplus H_2)\right)\otimes C(X)\right)
$$
by using the following formula:
$$
S:={\left(\begin{array}{cc} I- S_1S_1^*+S_1R_1^*S_1^* \\  R_2S_1^*\end{array}\right)}.
$$
It is a straightforward computation to show that $S$ is a unitary  and that it intertwines $\what\pi_2^\infty\oplus 0$ and $\what\pi_2^\infty\oplus 0 \oplus\what\pi_1^\infty\oplus 0$ up to compacts. The operator $R_1$ commutes with the representation $\id_{\ell^2\N}\otimes \what\pi^\infty_1\oplus 0$, and $R_2$ intertwines (exactly) the representations
$\what\pi_1^\infty\oplus 0$ and $(\id_{\ell^2\N}\otimes \what\pi_1^\infty)\oplus 0$, and therefore have {support contained in the diagonal of $Z^2$. Whence, the operator $S$ is localized by composition with the propagation  index which is $\leq 7$. {We refer the reader again to Appendix \ref{LocalizedOperators} for the properties of localized operators.}
%contained in the (closure of) the subspace
%$$
%\bigcup_{(g,g')\in \Gamma_\chi^{(1)}} \; \left[g\Supp(\chi)\times g'\Supp(\chi)\right].
%$$
Again in the cocompact case and with the $\Gamma$-invariant distance on $Z$, we see that the operator $S$ has finite propagation which is $\leq$ to the diameter of the compact space $\cup_{g\vert g\Supp(\chi)\cap\Supp(\chi)\neq\emptyset}\;   \Supp (g\chi)$. Indeed, this is a finite union by definition of the cutoff function.}

%
%All other operators appearing in the formula for $S$ have finite propagation,
%therefore so does $S$.
A similar unitary $T$ exists between $\what\pi_1^\infty\oplus 0$ and $\what\pi_1^\infty\oplus 0 \oplus\what\pi_2^\infty\oplus 0$ intertwining them up to compacts {with the same condition on its support and hence with finite propagation in the cocompact metric case}. After applying a suitable flip
unitary $\alpha$ in the target space of these unitaries, which exchanges the first two and the last two factors (i.e. $\alpha$ is given by $(u_1, u_2, u_3, u_4)\mapsto (u_3, u_4, u_1, u_2)$),
and then taking the composition $U=T^*\alpha^*S$, one gets the desired
unitary which intertwines $\what\pi_1^\infty\oplus 0$ and $\what\pi_2^\infty\oplus 0$ up to compacts. One can check as well  that the flip unitary $\alpha$ has support inside the diagonal of $Z^2$. {In conclusion, the unitary $U$ is also localized with the same propagation index, which is hence $\leq 7$.}
As a consequence in the cocompact case with the $\Gamma$-invariant distance, we conclude again by an easy verification that the unitary $U$ has finite propagation as desired.
% We also need a variant of the previous lemma in the proof of Paschke duality, in the case when the representations $\what\pi_i$ on $H\otimes C(X)$ themselves are extensions by zero
% of $G$-equivariant representations on submodules $E_i$ of $H\otimes C(X)$.
%
% \begin{lemma}
%  Let $E_i$ be $G$-Hilbert $C(X)$-modules and $\what\pi_i$ be $G$-equivariant fibrewise ample representations of $C_0(Y)$ on $E_i$ for $i=1,2$, respectively. Then there exists a unitary
%  $S\in \maL_{C(X)}(\ell^2\Gamma^\infty\otimes (E_1\oplus E_2), \ell^2\Gamma^\infty\otimes (E_2\oplus E_1))$ such that $S$ has uniformly finite propagation in the $X$-variable, and
%  $$
%  S^*(\what\pi^\infty_2(f)\oplus 0)S- (\what\pi^\infty_1(f)\oplus 0) \in \maK_{C(X)}(\ell^2\Gamma^\infty\otimes (E_1\oplus E_2)) \quad \quad \text{ for all } f\in C_0(Y)
%  $$
%  where $\what\pi_i^\infty$ is the extended representation on $\ell^2\Gamma^\infty$ by tensoring with the identity on $\ell^2\Gamma^\infty$, and $\what\pi^\infty_1\oplus 0$ is the
%  extended represention on $\ell^2\Gamma^\infty\otimes (E_1\oplus E_2)$ and similarly for $\what\pi^\infty_2(f)\oplus 0$.
\end{proof}

Let us now take into account the action of our discrete countable group  $\Gamma$  by homeomorphisms on $X$. Recall that  $A$ is a proper $\Gamma$-algebra over $Z$ and that $C_0(Z)$ maps inside the center of $A$ itself.  We denote as before by $G$ the action groupoid $X\rtimes \Gamma$ or its space of arrows, since no confusion can occur.
A specific unitary representation of $\Gamma$ is the (right) regular representation $\rho$ in the Hilbert space $\ell^2\Gamma$, which can be tensored by the identity of
$\ell^2\N$ to get the unitary representation $\rho^\infty$ of $\Gamma$ in $\ell^2\Gamma^\infty=\ell^2\Gamma\otimes \ell^2\N$. {Recall that given the proper $\Gamma$-space, we always choose a uniform cutoff function $\chi\in C(Z)$, and that   the first (and also the second) projection $\Gamma_\chi^{(1)}\to \Gamma$ is proper.} We are now in position to state Theorem \ref{GequivVRiso2}, that we restate using the groupoid language, so as to fit with possible generalizations, as follows:

\medskip

\begin{theorem}\label{GequivVRiso1}
Assume that the action groupoid $G=X\rtimes \Gamma$ acts  properly on the $G$-space $Y=X\times Z$, meaning here that $\Gamma$ acts  properly on $Z$ with a {fixed uniform cutoff function $\chi\in C(Z)$.} Let $\what\pi_1$ and $\what\pi_2$ be two fiberwise ample $G$-equivariant representations of
$C(X, A)$ in the Hilbert $G$-modules $H_1\otimes C(X)$  and $H_2\otimes C(X)$ respectively.
Then, identifying each $\what\pi_i$ with the trivially extended representation $\left(\begin{array}{cc} \what\pi_i & 0\\ 0 & 0\end{array}\right)$ that is further tensored by the identity
of $\ell^2\Gamma^\infty$, there exists a $G$-invariant unitary operator
$$
W\in \maL_{C(X)} \left(\left[ \ell^2\Gamma^\infty\otimes (H_1\oplus H_2)\right] \otimes C(X), \left[\ell^2\Gamma^\infty \otimes (H_2\oplus H_1) \right]\otimes C(X)\right),
$$
 such that for any $\varphi\in C(X, A)$
$$
  W^*\what\pi_2 (\varphi) W - \what\pi_1(\varphi) \in \maK_{C(X)} \left( [\ell^2\Gamma^\infty\otimes (H_1\oplus H_2) ]\otimes C(X)\right).
  $$
Moreover, {we can ensure that the operator $W$ is localized with the propagation index $\leq 7$.
In particular,} if $Z/\Gamma$ is compact with the previous metric assumption on $Z$, then we can ensure that $W$  has finite propagation.
\end{theorem}

% \begin{remark}
% In the proof of the Paschke duality theorem below, we shall only use that $W$ is an isometry.
% \end{remark}

\begin{proof}
Since the extended representations (of the unitalization $C(X, A^+)$) are fiberwise ample (say homogeneous in the terminology used in \cite{PPV}),
by ``forgetting'' the right regular $\Gamma$-action on $\ell^2\Gamma^\infty$, from Theorem \ref{NCVRisometry2} we deduce again the existence of a unitary {that we rather denote in this proof by $S$ ($U$ will denote below another family of isometries) with support within  $\maA_7$ (so with
 finite propagation when $Z/\Gamma$ is compact and $Z$ is metric-proper):}
$$
S \in \maL_{C(X)} \left(\left[ \ell^2\Gamma^\infty\otimes (H_1\oplus H_2)\right] \otimes C(X), \left[\ell^2\Gamma^\infty \otimes (H_2\oplus H_1) \right]\otimes C(X)\right)
$$
such that {for any $f\in C(X, A^+)$},
$$
S^*\what{\pi}_2 (f) S - \what{\pi}_1(f) \in  \maK_{C(X)} \left( [\ell^2\Gamma^\infty\otimes (H_1\oplus H_2) ]\otimes C(X)\right).
$$
In particular this property holds for the restrictions of $\what\pi_1$ and $\what\pi_2$ to $C(X, A)$.

The unitary $U$ obtained in this way is of course a priori not $\Gamma$-invariant. To remedy this, we shall use a classical trick which allows to ``average''. Using Fell's trick, one can construct a family of   operators $(U_{g})_{g\in \Gamma}$ acting on $\ell^2(\Gamma)^\infty$,
such that (see for instance \cite{GWY16} or \cite{BR2}): %or \cite{BR-Paschke} where the proof is given for all \'etale groupoids):

\begin{itemize}
 \item for $g, g'\in \Gamma$, $U_{g}^*U_{g'} = \delta_{g,g'}\id_{\ell^2(\Gamma)^\infty}$, in particular each $U_g$ is an isometry;
 \item $\sum_{g\in \Gamma} U_{g}U_{g}^* = \id_{\ell^2(\Gamma)^\infty}$; and
 \item ($\Gamma$-equivariance) $U_{g' g}\;= \; \rho_{g}^\infty \; U_{g'}\; \rho_{g^{-1}}^\infty$ for any $(g,g')\in \Gamma^2$.
\end{itemize}

{Here of course $\rho$ is the right regular representation of $\Gamma$.
Recall the cutoff funtion $\chi\in C(Z)$ defined using the properness  of the $\Gamma$-action  on  $Z$} and which is compactly supported when $Z/\Gamma$ is assumed compact. We proceed now to  define the  allowed field $W_x: (H_1\oplus H_2)\otimes \ell^2(\Gamma)^\infty \longrightarrow (H_2\oplus H_1)\otimes \ell^2(\Gamma)^\infty$ or equivalently the corresponding operator $W$ obtained by the averaging trick.

Consider the dense submodule $\maE'_1$ of $\left[ \ell^2\Gamma^\infty\otimes (H_1\oplus H_2)\right] \otimes C(X)$  which is given by
$$
\maE'_1\;  := \; ({\what\pi_1} \oplus {\what\pi_2}) (C(X, A_c)) \left( \left[ \ell^2\Gamma^\infty\otimes (H_1\oplus H_2)\right] \otimes C(X) \right)
$$
In this notation, $A_c=C_c(Z) A$ as before, and ${\what\pi_i}$ is the original representation of $C(X, A)$ on $C(X)\otimes H_i$ that we have tensored with the identity in $\ell^2\Gamma^\infty$. We  similarly define $\maE'_2$.

Notice that, ${\what\pi}_i$ also denotes the extended representation of $C(X, A)$ in $\left[ \ell^2\Gamma^\infty\otimes (H_1\oplus H_2)\right] \otimes C(X)$
(resp. $\left[ \ell^2\Gamma^\infty\otimes (H_2\oplus H_1)\right] \otimes C(X)$) obtained as ${\what\pi}_i\oplus 0$ for $i=1,2$, respectively.
On the other hand,  an operator
$$
T\in \maL_{C(X)}\left( \left[ \ell^2\Gamma^\infty\otimes (H_1\oplus H_2)\right] \otimes C(X), \left[ \ell^2\Gamma^\infty\otimes (H_2\oplus H_1)\right] \otimes C(X)\right)
$$
is $G$-invariant if for any $g\in \Gamma$, we have {for $(x, g)\in X\times \Gamma$:
$$
T_x = (gT)_x := (V^2_{(x,g)}\oplus V^1_{(x,g)}
) T_{xg} (V^1_{(x,g)}\oplus V^2_{(x, g)})^{-1}.
$$}
Recall that $V^i$ denotes the extensions of the $X\rtimes\Gamma$-actions on $H_i\otimes C(X)$ by tensoring with the right regular representation of $\Gamma$, $\rho^\infty$ on $\ell^2\Gamma^\infty$.
So, {choosing a cutoff function $\chi$ as before with the extra property that the first projection $\Gamma_\chi^{(1)}\to \Gamma$ is proper}, we now replace $S$ by the (well defined) average operator
$$
W = \sum_{g\in \Gamma}(\id_{H_2\oplus H_1}\otimes U_{g}) \circ g \left(S \pi_{1} (\sqrt{\chi})\right).
$$
Here  $(U_g)_{g\in \Gamma}$ is the family of isometric operators on the Hilbert space $\ell^2(\Gamma)^\infty$ defined above. For $e\in \maE'_1$, we thus have defined
$$
W_x (e_x) := \sum_{g\in \Gamma} \left(\id_{H_2\oplus H_1}\otimes U_{g}\right) \, \left(V^2\oplus V^1\right)_{(x,g)}  S_{xg} \left(V^1\oplus V^2\right)_{(xg,g^{-1})} \pi_{1,x} (g^*\sqrt{\chi})  (e_x)
$$
The sum defining $W_x$ is then finite since for any $\varphi\in C_c(Z)$ that is viewed in $A$, we have
$$
(\pi_{1,x}\oplus 0) (g^*\sqrt{\chi})  (\pi_{1,x} \oplus \pi_{2, x}) (\varphi) = (\pi_{1,x} (\varphi \sqrt{g^*\chi})\oplus 0),
$$
and the number of $g\in \Gamma$ such that $\varphi \sqrt{g^*\chi}\neq 0$ is finite by the properness of the $\Gamma$-action on $Z$. Hence, ${W (e)}$ is well defined on the elements  $e\in \maE'_1$. Moreover, an easy inspection, using the properties of the family $(U_g)_{g\in \Gamma}$, shows that the relation $
W_x^* W_x = \id$ holds on   $\maE'_1$.

%More precisely,
% \begin{eqnarray*}
%  <W_xe_x, W_xe_x>&=& \sum_{g,g'\in \Gamma}<(U_{(x,g)}\otimes \id)V^2_{(x,g)} S_{xg}\pi^\infty_{1,x}(\sqrt{c})(V^1_{(x,g)})^{-1}e_x, \\
%  &&\quad(U_{(x,g')}\otimes \id)V^2_{(x,g')}S_{xg'}\pi^\infty_{1,x}(\sqrt{c})(V^1_{(x,g')})^{-1}e_x>\\
%  &=& \sum_{g\in \Gamma}<(U_{(x,g)}\otimes \id)V^2_{(x,g)} S_{xg}\pi^\infty_{1,x}(\sqrt{c})(V^1_{(x,g)})^{-1}e_x,\\
%  &&\quad(U_{(x,g)}\otimes \id)V^2_{(x,g)} S_{xg}\pi^\infty_{1,x}(\sqrt{c})(V^1_{(x,g)})^{-1}e_x>\\
%  &=& \sum_{g\in \Gamma}<V^1_{(x,g)} \pi^\infty_{1,x}(c) (V^1_{(x,g)})^{-1}e_x,e_x> \quad\text{ (using the fact the $U_{(x,g)}, V^2_{(x,g)}, S_{xg}$ are isometries)}\\
%  &=& <e_x,e_x> \quad\text{ (using the property of the cutoff function $c$)}
% \end{eqnarray*}
This shows that $W_x$ automatically extends to an isometry  between the corresponding Hilbert spaces that we still denote $W_x$.
Moreover, when $e=[\what\pi_1 (f)\oplus  \what\pi_2 (f) ]e_1$ with $f\in C(X, A_c)$,  there is a  finite subset $I_e$ of $\Gamma$, {which does not depend on the variable $x\in X$},  such that,
$$
W (e)_x= \sum_{g\in I_e} T_{g, x}(e_x)\quad \forall x\in X.
$$
Here each of the maps $x\mapsto T_{g, x} (e_x)$ and $x\mapsto T_{g, x}^* (e_x)$ is of course norm-continuous.
We thus end up with the adjointable isometry, still denoted $W$,
between the Hilbert modules $\left[(H_1\oplus H_2)\otimes \ell^2\Gamma^\infty\right]\otimes C(X)$  and $\left[(H_2\oplus H_1)\otimes \ell^2\Gamma^\infty\right]\otimes C(X)$ as announced. Notice that in the cocompact case, $W$ has finite propagation by construction.

Now, $W$ satisfies the following properties:
\begin{enumerate}
% \item the map $x\mapsto W_x$ is $*$-strongly continuous;
 \item $W^*\what\pi_{2}(f)W- \what\pi_{1}(f) \in \maK_{C(X)}\left([(H_1\oplus H_2)\otimes \ell^2(\Gamma)^\infty]\otimes C(X) \right)$, for any $f\in C(X, A_c)$;
  \item $W$ is $G$-invariant.
\end{enumerate}
Once these properties have been verified, a standard trick as in the proof of Theorem \ref{NCVRisometry2} using the direct sum of the representations allows to find in place of the isometry $W$,
a unitary which will also satisfy the same two properties. Note  (see the notation in the proof of Theorem \ref{NCVRisometry2}) that if the initial isometry $S_0$ is {$G$-invariant, then the operators $S_1$, as well as $R_1$ and $R_2$ appearing in the proof of Theorem \ref{NCVRisometry2} are all $G$-invariant}
by construction.

Regarding the first item, notice that we have for any $f\in C(X, A_c)$:
$$
\what\pi_{2}(f)W = \sum_{g\in \Gamma} \left(\id_{H_2\oplus H_1}\otimes U_{g}\right)\circ  g \left(\what\pi_{2}(g^{-1}f) S \pi_{1} (\sqrt{\chi})\right),
$$
But an easy inspection of the support of the operator $S$ {(which is contained in $\maA_7$), using that the action is uniformly proper,} we deduce that $\what\pi_{2}(g^{-1}f) S \pi_{1} (\sqrt{\chi})$ is only non-zero for a finite number of elements $g\in \Gamma$. {The similar statement holds for $W\what\pi_1(f)$. An ad hoc consequence is that the support of $W$ is also contained in $\maA_7$.} When $Z/\Gamma$ is compact,  this is more obvious since the operator $S$ has  finite propagation.  Hence the sum defining the operator $\what\pi_{2}(f) W$ is finite
independently of the test vector $e$ and  therefore makes sense in the uniform operator topology.

Therefore, we may compute using the $G$-equivariance of the representations $\what\pi_i$:
\begin{eqnarray*}
\what\pi_{2}(f)W&=& \sum_{g\in \Gamma}\left(\id_{H_2\oplus H_1}\otimes U_{g}\right)\circ  g \left(\what\pi_{2}(g^{-1}f) S \pi_{1} (\sqrt{\chi})\right)\\
& \sim &  \sum_{g\in \Gamma}\left(\id_{H_2\oplus H_1}\otimes U_{g}\right) \circ g \left(S \what\pi_{1} (g^{-1}f)  \pi_{1} (\sqrt{\chi})\right)\\
& = & \sum_{g\in \Gamma} \left(\id_{H_2\oplus H_1}\otimes U_{g}\right) \circ g \left(S  {\pi_{1} (\sqrt{\chi})}\right) \circ g \what\pi_{1} (g^{-1}f)\\
& = & W \what\pi_{1} (f).
%\sum_{g\in \Gamma} V^1_{(x,g)}\what\pi_{1,x}(\sqrt{p_2^*c}) S^*_{xg} V^2_{(xg,g^{-1})} \what\pi_{2,x}(f) V^2_{(x,g)}S_{xg} \what\pi_{1,x}(\sqrt{p_2^*c})  V^1_{(xg,g^{-1})}\\
%&=&   \sum_{g\in \Gamma} V^1_{(x,g)}\what\pi_{1,x}(\sqrt{p_2^*c}) S^*_{xg}\what\pi_{2,x}(gf) S_{xg}\what\pi_{1,x}(\sqrt{p_2^*c}) V^1_{(xg,g^{-1})}\\
%%&=&   \sum_{g\in \Gamma} V^1_{(x,g)}(\pi_Y(\sqrt{c} (S^*_{xg}\pi(gf)S_{xg}
%%(\pi_Y(\sqrt{c} ) (V^1_{(x,g)})^{-1}\\
%&\sim&  \sum_{g\in \Gamma} V^1_{(x,g)}\what\pi_{1,x}(\sqrt{p_2^*c})\what\pi_{1,x}(gf) \what\pi_{1,x}(\sqrt{p_2^*c} ) V^1_{(xg,g^{-1})}\\
%&=& \sum_{g\in \Gamma} V^1_{(x,g)}\what\pi_{1,x}(gf) V^1_{(xg,g^{-1})}
%V^1_{(x,g)}\what\pi_{1,x}(p_2^*c ) V^1_{(xg,g^{-1})}\\
%&=& \sum_{g\in \Gamma} \what\pi_{1,x}(f)
%V^1_{(x,g)}\what\pi_{1,x}(p_2^*c ) (V^1_{(x,g)})^{-1}\\
%&=& \what\pi_{1,x}(f)
\end{eqnarray*}
The sign $\sim$ again refers to equality modulo the compact operators of the corresponding Hilbert modules and since the sum is finite,
the operator $\what\pi_{2}(f)W - W \what\pi_{1} (f)$ is clearly compact. Now, since $W^*$ is an adjointable operator, composing with $W^*$ on the left yields to the conclusion.

Finally, $W$ was indeed born to be $G$-invariant.
Since the submodule $\maE'_1$ is a $G$-submodule, we may prove $G$-invariance strongly on the vectors of $\maE'_1$.
Let us denote the $G$-actions on $H_1\oplus H_2$ by $\what{V}^1:= V^1\oplus V^2$ and similarly by $\what{V}^2:= V^2\oplus V^1$ the $G$-action on $H_2\oplus H_1$.
We then compute for any $(x, h)\in G$:
\begin{eqnarray*}
 W_x\what{V}_{(x,h)}^1 &=& \sum_{g\in \Gamma} (\id\otimes  U_{g})\what{V}^2_{(x,g)} S_{xg}\pi_{1,x}\left(\sqrt{\chi}\right) \what{V}^1_{(xg,g^{-1})} \what{V}^1_{(x,h)}\\
 &=& \sum_{g\in \Gamma} (\id\otimes  U_{g})\what{V}^2_{(x,g)} S_{xg}\pi_{1,x}\left(\sqrt{\chi}\right) \what{V}^1_{(xg,g^{-1}h)}\\
 &=& \sum_{l\in \Gamma}  (\id\otimes U_{hl})\what{V}^2_{(x,hl)}S_{xhl}\pi_{1,x}\left(\sqrt{\chi}\right) \what{V}^1_{(xhl,l^{-1})}
 \end{eqnarray*}
 Hence
 \begin{eqnarray*}
W_x\what{V}_{(x,h)}^1 &=& \sum_{l\in \Gamma} (\id\otimes  U_{hl})\what{V}^2_{(x,h)}\what{V}^2_{(xh,l)} S_{xhl}\pi_{1,x}\left(\sqrt{\chi}\right) (\what{V}^1_{(xh,l)})^{-1}\\
 &=& \sum_{l\in \Gamma} \what{V}^2_{(x,h)}(\id\otimes U_{l})\what{V}^2_{(xh,l)}S_{xhl}\pi_{1,x}\left(\sqrt{\chi}\right) (\what{V}^1_{(xh,l)})^{-1}\\
 &=& \what{V}^2_{(x,h)} W_{xh}
\end{eqnarray*}
which is the required right hand side.
\end{proof}

It is worthpointing out that all the previous theorems apply to the case of $A=C_0(Z, B)$ where $B$ is any separable unital $\Gamma$-algebra. An already interesting application is when $A= C_0(Z)$ as we shall see in the next section.

\medskip

\subsection{The general case}\label{ControlNow}

By using an easy generalization of the PPV work, expanded in Appendix \ref{ControlPPV}, we now state the norm-controlled version of our main Theorem \ref{GequivVRiso2}, say Theorem \ref{GequivVRiso} which gives the precise generalization of results in \cite{Voiculescu}, compare also with \cite{Kasparov}. So the goal of this section is to explain how to adapt the proof of the previous section so as to construct the sequence of unitaries of Theorem  \ref{GequivVRiso}.
{For a family $T=(T_\ep)_{0<\ep \leq \ep_0}$ of operators $T_\ep\in C(X, {\maK(H)})$, we shall use the notation $T\stackrel{\epsilon}{\sim} 0$
to denote the fact that the family $T$ (is composed of compact operators $T_\ep$ which) have {uniform} norm at most a constant multiple of $\epsilon$ as $\ep\to 0$; the constant may depend on the family $T$. More precisely, such a family $T:=(T_\ep)_{0< \ep \leq\ep_0}$  in $C(X, {\maK(H)})$ satisfies $T\stackrel{\epsilon}{\sim} 0$ if there is some constant $C(T) >0$ such that  $\vert \vert T_\ep\vert\vert  \leq C(T) \ep$ for any $\ep$.}
Recall that $A$ is a separable proper $\Gamma$-algebra over $Z$.

Let $\Sigma$ be a countable dense subset of the separable $C^*$-algebra $C(X, A)${, which contains $0$,} is closed under the involution $a\mapsto a^*$, and globally $\Gamma$-invariant. Such $\Sigma$ always exists since we can for instance take the union of the $\Gamma$-orbits  of a countable dense self-adjoint subset of $C(X, A)$.

\begin{theorem}\label{GequivControlledPPV}[Controlled version of Theorem \ref{GequivVRiso2}]\\
 Under the assumptions of Theorem \ref{GequivVRiso2}, if we fix $\ep >0$ then we can ensure that  the $\Gamma$-invariant  unitary operator
$$
W=W_\ep\in \maL_{C(X)} \left([(H_1\oplus H_2)\otimes \ell^2\Gamma^\infty]\otimes C(X) ,  [(H_2\oplus H_1)\otimes \ell^2\Gamma^\infty] \otimes C(X)\right),
$$
obtained in that theorem, satisfies in addition the following control condition:
$$
\forall \varphi \in \Sigma, \exists C_\varphi \text{ independent of $\ep$ such that } \vert\vert W_\ep^*\what\pi_2 (\varphi) W_\ep - \what\pi_1(\varphi) \vert\vert \leq C_\varphi \ep.
$$
\end{theorem}
\medskip

Said differently, $W_\ep$ satisfies the {support condition} plus the relation
$$
  W_\ep^*\what\pi_2 (\varphi) W_\ep - \what\pi_1(\varphi) \stackrel{\epsilon}{\sim} 0, \quad \text{ for all } \varphi\in \Sigma. %C_0(X\times Z).
  $$
\medskip

We only need to explain how to complete  the proof given for Theorem \ref{GequivVRiso2} so that the control is ensured. We thus start by stating the following Lemma which generalizes Lemma \ref{finpropVRiso1}.

\medskip

\begin{lemma}\label{ControlledPPViso1}[Controlled version of Lemma \ref{finpropVRiso1}]\\
 {{Under the assumptions of Theorem \ref{GequivControlledPPV}}} and given $\ep >0$, there exists an isometry
 $$
 \hat{S}_\ep \in \maL_{C(X)}((H_1\oplus H_2)\otimes C(X),  (H_2\oplus H_1)\otimes \ell^2\Gamma \otimes C(X))
 $$
 with  the same support  as in Lemma \ref{finpropVRiso1} such that
 $$
 \hat{S}_\ep^*((\what\pi_2(\varphi)\otimes \id_{\ell^2\Gamma})\oplus 0)\hat{S}_\ep- (\what\pi_1(\varphi)\oplus 0) \stackrel{\epsilon}{\sim} 0, \quad  \forall \varphi \in \Sigma.
 $$
\end{lemma}

\medskip

\begin{proof}\
We explain the needed complements to  the steps in the proof of Lemma \ref{finpropVRiso1}, which  exploit the norm control on the residual compact operators as in Theorem \ref{PPVseq} of the appendix. We also forget the support condition for $\hat{S}$ which is again satisfied as one can check easily. We again
fix a cut-off function $\chi\in C(Z)$. Let  $V_\chi$ is the {{non-empty}} interior of the support of $\chi$.

\textbf{Step 1:} We need to apply Corollary \ref{PPVunitaryseq}.
Let $\Lambda$ be a countable dense subset of $C(X)$, containing $0$ and $1_X$, which is closed under adjoints. {Let $\Sigma^0$ be a countable dense subset of $C(X, C_0(Z, \chi)A^+)$ which contains $\Sigma$, is closed under adjoints and such that  (modifying  $\Sigma$ if necessary)   $\Sigma^0\cap C(X, A_\chi)$ is dense in $C(X,A_\chi)$. This latter condition can be easily achieved. One can  for instance  choose an extra countable dense self-adjoint subset $\Sigma^1$ of $C(X, A_\chi)$ and  replace $\Sigma^0$ by $\Sigma^0\cup \Sigma^1$. Alternatively  one can join to $\Sigma$ from the beginning the self-adjoint $\Gamma$-subset of all $\Gamma$-orbits of  some dense countable self-adjont subset of $C(X, C_0(V_\chi)A)$, and then easily build up $\Sigma^0$ with the above properties.  Consider then the subset
$\Sigma^0_\chi:=\{r_n\}_{n\in \N}$ of $C(X, A_\chi)\oplus C(X)$ composed of elements $r_n$ which either belong to $(\Sigma^0\cap C(X, A_\chi), \Lambda)$ or are of the form $(f\chi^{1/2},0)$,
where $f\in \Sigma$.} Since the sequence
$$
\{r'_n= r_n/(1+n.||r_n||)\,\text{ for } r_n\in \Sigma^0_\chi\}_{n=1}^{\infty}
$$
is convergent to $0$ in $C(X, A_\chi)\oplus C(X)$, the collection $\Sigma_\chi=\{(\chi^{1/2},0), (0,1_X)\}\cup \{r'_n\}_{n\in \N}$ is a compact self-adjoint total subset of
$C(X, A_\chi)\oplus C(X)$.

Then the image $(\what{\pi}^\chi_1\oplus 0)^+(\Sigma_\chi)$ is a self-adjoint compact subset of
$C_\chi:= (\what{\pi}^\chi_1\oplus 0)^+(C(X, A_\chi)\oplus C(X))$, which is total in $C_\chi$ and contains the identity. Let $B_\chi$ be the algebra generated by $C_\chi$ and $C(X, \maK(H_1))$, which defines an $X$-extension algebra for the
unital algebra
$C(X, A_\chi)\oplus C(X)$. Since $(\what{\pi}_1^\chi\oplus 0)^+ $ is fibrewise ample, using the same arguments as in the proof of Lemma \ref{finpropVRiso1}, we conclude
that $B_\chi$ is a trivial $X$-extension. Using the separability of $C(X, \maK(H_1))$, we fix a compact self-adjoint total subset $F_\chi$ of $B_\chi$ which contains $\Sigma_\chi$.

Consider also the trivial $X$-extension obtained analogously by $(\what{\pi}_2^\chi\oplus 0)^+$. Then we get, using the notations in the proof
of Lemma \ref{finpropVRiso1} and Corollary \ref{PPVunitaryseq} (for the compact subset $F_\chi$), a unitary $S_{\chi, \ep}\in \maL_{C(X)} ({(H_1\oplus H_2)}\otimes C(X), {(H_2\oplus H_1)}\otimes C(X))$ depending on $\ep$, such that we have in particular for any
$\tilde{f}\in \Sigma_\chi$ of the form $\tilde{f}=(f, 0)$:
$$
S^*_{\chi, \ep} \begin{bmatrix}   \what\pi^\chi_2(f) &0\\ 0 & 0   \end{bmatrix} S_{\chi, \ep} -\begin{bmatrix}   \what\pi^\chi_1(f) &0\\ 0 & 0   \end{bmatrix} \stackrel{\epsilon}{\sim} 0
$$
where the Hilbert $C(X)$-modules $H_i\otimes C(X)$ {are} decomposed as $\maE^i_\chi\oplus (H_i\otimes C(X))$ for $i=1, 2$.\\

\textbf{Step 2:} Let $s_{\chi, \ep}: \maE^1_\chi \rightarrow \maE^2_\chi$ be the $(1,1)$-entry in the matrix decomposition of $S_{\chi, \ep}$. Then we have for $\tilde{f}=(f,0)\in \Sigma_\chi$:
\begin{enumerate}
\item $s^*_{\chi, \ep} \what\pi^\chi_2(f)s_{\chi, \ep}- \what\pi^\chi_1(f) \stackrel{\epsilon}{\sim} 0$, %\quad S_{12}^*\what\pi^\chi_{Y}^\infty(f) \sim 0, \quad \what\pi^\chi_Y^\infty(f)S_{12} \sim 0$
\item $\what\pi^\chi_2(f)(s_{\chi, \ep} s^*_{\chi, \ep}- \id) \stackrel{\epsilon}{\sim} 0$ and  $(s^*_{\chi, \ep} s_{\chi, \ep}- \id)\what\pi^\chi_1(f)\stackrel{\epsilon}{\sim} 0$,
%\item $\what\pi^\chi_Y^\infty(f)(ss^*s- s) \sim 0, \quad (s^*ss^*- s^*)\what\pi^\chi_Y^\infty(f)\sim 0$
\item $[s_{\chi, \ep} s^*_{\chi,\ep} ,\what\pi^\chi_2(f)]\stackrel{\epsilon}{\sim} 0$,
%\item $[s(F\oplus \id)s^*, \what\pi^\chi_Y^\infty(f)]\sim 0$, and
%\item $\what\pi^\infty_Y(f)((s(F\oplus \id)s^*)^2- ss^*)\sim 0$
\item $(1-s^*_{\chi, \ep} s_{\chi, \ep})$ and $(1-s_{\chi, \ep} s^*_{\chi, \ep})$ are positive operators.
\end{enumerate}

Indeed, for the first item, it suffices to observe that $s^*_{\chi, \ep}\what\pi^\chi_2(f)s_{\chi, \ep}- \what\pi^\chi_1(f)$ is the $(1,1)$-entry in the matrix given by:
$$
S^*_{\chi, \ep}\begin{bmatrix}   \what\pi^\chi_2(f) &0\\ 0 & 0   \end{bmatrix} S_{\chi, \ep} -\begin{bmatrix}   \what\pi^\chi_1(f) &0\\ 0 & 0   \end{bmatrix}
$$
Since the norms of the elements constituting a $2\times 2$ matrix is bounded above by the norm of the matrix itself, we are done. The proof of the other properties {{is}} similar.
As in Lemma \ref{finpropVRiso1} we define the unitary $\hat{s}_{\chi, \ep}: \maE^1_\chi\oplus \maE^2_\chi\rightarrow \maE^2_\chi\oplus \maE^1_\chi$ as follows:
$$
\hat{s}_{\chi, \ep}:= \begin{bmatrix}   s_{\chi, \ep} & (1-s_{\chi, \ep} s_{\chi, \ep}^*)^{1/2}\\ {-}(1-s_{\chi, \ep}^* s_{\chi, \ep})^{1/2} & s_{\chi, \ep}^*   \end{bmatrix}
$$
We have the following formula for any $(f,0)\in \Sigma_\chi$:
$$
\hat{s}^*_{\chi, \ep} (\what\pi^\chi_2(f)\oplus 0)\hat{s}_{\chi, \ep}- (\what\pi^\chi_1(f)\oplus 0) =
\begin{bmatrix}   s^*_{\chi, \ep} \what{\pi}^\chi_2(f)s_{\chi, \ep} -\what{\pi}^\chi_1(f) & s_{\chi, \ep}^*\what{\pi}_2^\chi(f)(1-s_{\chi, \ep} s_{\chi, \ep}^*)^{1/2}\\ (1-s_{\chi, \ep} s^*_{\chi, \ep})^{1/2}\what{\pi}_2^\chi(f)s_{\chi, \ep} & (1-s_{\chi, \ep} s^*_{\chi, \ep})^{1/2}\what{\pi}_2^\chi(f)(1-s_{\chi, \ep} s^*_{\chi, \ep})^{1/2}  \end{bmatrix}
$$
Note that $||s_{\chi, \ep}||\leq 1$, and $||(1-s_{\chi, \ep} s_{\chi, \ep}^*)^{1/2}||\leq 1$. We also have
$$
||(1-s_{\chi, \ep} s_{\chi, \ep}^*)^{1/2}\what{\pi}^\chi_2(f)||^2= ||\what{\pi}^\chi_2(f^*)(1-s_{\chi, \ep} s_{\chi, \ep}^*)\what{\pi}^\chi_2(f)||\leq ||f||.||(1-s_{\chi, \ep} s_{\chi, \ep}^*)\what{\pi}^\chi_2(f)||
$$

Thus we get from properties (2) and (3) above that $(1-s_{\chi,\ep} s_{\chi, \ep}^*)^{1/2}\what{\pi}^\chi_2(f) \stackrel{\epsilon}{\sim} 0$. Therefore all the matrix entries $A_{ij}$ in the above matrix
satisfy $A_{ij}\stackrel{\epsilon}{\sim} 0$. Thus we get:
$$
\hat{s}^*_{\chi, \ep} (\what\pi^\chi_2(f)\oplus 0)\hat{s}_{\chi, \ep}- (\what\pi^\chi_1(f)\oplus 0) \stackrel{\epsilon}{\sim} 0
$$

Extending the unitary $\hat{s}_{\chi, \ep}$ by zero, we  get a partial isometry in
$\maL_{C(X)}((H_1\oplus H_2)\otimes C(X), (H_2\oplus H_1)\otimes C(X))$, that we still denote by $\hat{s}_{\chi, \ep}$. \\

\textbf{Step 3 :} For  $(x,g)\in G$, we denote by $V^i_{(x,g)}$  the unitary implementing the $G$-action on $H_i\otimes C(X)$.
Then we define an operator $\hat{S}_\ep \in \maL_{C(X)}((H_1\oplus H_2)\otimes C(X), \ell^2\Gamma\otimes (H_2\oplus H_1)\otimes C(X))$ {{by setting  the following pointwise formula:
$$
\left[\hat{S}_{\ep, x}\right]_g:= \left(V^2_{(x,g)} \oplus V^1_{(x,g)}\right) \hat{s}_{\chi, \ep} \left(\what\pi_1(\chi^{1/2})\oplus \what\pi_2(\chi^{1/2})\right)
\left(V^1_{(x,g)} \oplus V^2_{(x,g)}\right)^{-1}.
$$}}
%\begin{lemma}\label{finitepropVRiso}
The operator $\hat{S}_\ep$ then satisfies the allowed properties in the statement of Lemma \ref{ControlledPPViso1}, as we prove below. Let us show that we also have:
$$
 \hat{S}_\ep^*((\what\pi_2(f)\otimes \id_{\ell^2\Gamma})\oplus 0)\hat{S}_\ep- (\what\pi_1(f)\oplus 0) \stackrel{\epsilon}{\sim} 0, \quad  \forall f\in \Sigma.
$$
{Replacing $f$ by a compactly supported element which is as uniformly close as we please to $f$, we may assume that $f$ is itself compactly supported.}
Denote then by $\Gamma(\chi, f)$ the set of $g\in \Gamma$ such that $\Supp(g^*f)\cap \Supp(\chi)\neq \emptyset$. Due to the properness of the $\Gamma$-action, this is a finite set.
Consider the compact operators for $g\in \Gamma, f\in \Sigma$:
$$
K_{\chi, \ep}:=\hat{s}_{\chi, \ep} \left(\what\pi_1(\chi^{1/2})\oplus \what\pi_2(\chi^{1/2})\right)- \left(\what\pi_2(\chi^{1/2})\oplus \what\pi_1(\chi^{1/2})\right)\hat{s}_{\chi, \ep}, \quad \text{and}
$$
$$
K_{\ep, g}(\chi, f)=\hat{s}^*_{\chi, \ep} (\what\pi^\chi_2((g^*f) \chi^{1/2})\oplus 0)\hat{s}_{\chi, \ep}- (\what\pi^\chi_1((g^*f)\chi^{1/2})\oplus 0)
$$

Note that we have $||K_{\ep, g}(\chi, f)||\leq C_g\epsilon$ for some constant $C_g>0$
independent of $\epsilon$ and similarly for $K_{\chi, \ep}$ with constant $C_\chi$.

Then from the computation in item (2) in the proof of Lemma \ref{finpropVRiso3}, we get:
\begin{eqnarray*}
&&[\hat{S}_\ep^*((\what\pi_2(f)\otimes \id_{\ell^2\Gamma})\oplus 0)\hat{S}_\ep- (\what\pi_1(f)\oplus 0)]_{x} \\
&=&{\sum_{g\in \Gamma(\chi, f)} \left(V^1_{(x,g)} \oplus V^2_{(x,g)}\right) \left(\what\pi_1(\chi^{1/2})\oplus \what\pi_2(\chi^{1/2})\right) \left({\hat s}_{\chi, \ep}^* ({\hat\pi}_2(g^*f)\oplus 0) K_\chi+K_g(\chi, f)\right)_x
\left(V^1_{(x,g)} \oplus V^2_{(x,g)}\right)^{-1}}
\end{eqnarray*}
{Recall  that when $\supp (\varphi)\cap \supp (\chi) =\emptyset$ we have ${\hat s}_{\chi, \ep}^* ({\hat\pi}_2(\varphi)\oplus 0) =0$ since the range of ${\hat s}_{\chi, \ep}$ is contained in $\maE_\chi^2\oplus \maE_\chi^1$.}

Thus one gets for any $f\in \Sigma$,
$$
||\hat{S}_\ep^*((\what\pi_2(f)\otimes \id_{\ell^2\Gamma})\oplus 0)\hat{S}_\ep- (\what\pi_1(f)\oplus 0)||\leq \left(C_\chi+\max_{g\in \Gamma(\chi, f)}C_g\right). |\Gamma(\chi,f)|.\epsilon
$$
This proves the claim.
\end{proof}

\begin{corollary}[Norm-controlled version of Theorem \ref{NCVRisometry2}]\label{ControlledPPViso2}
  There exists a unitary
 $$
 \hat{S}_\ep \in \maL_{C(X)}(\ell^2\Gamma^\infty\otimes (H_1\oplus H_2)\otimes C(X), \ell^2\Gamma^\infty\otimes (H_2\oplus H_1)\otimes C(X))
 $$
as in Theorem \ref{NCVRisometry2}  such that $
 \hat{S}_\ep^*((\what\pi_2(f)\otimes \id_{\ell^2\Gamma})\oplus 0)\hat{S}_\ep- (\what\pi_1(f)\oplus 0) \stackrel{\epsilon}{\sim} 0$ for any  $f\in \Sigma.$
\end{corollary}

\begin{proof}
 By directly verifying the constructions in Corollary \ref{balance} and the proof of Theorem \ref{NCVRisometry2}, we see that if the initial isometry is chosen to satisfy the conditions in
 Lemma \ref{ControlledPPViso1}, then all the intertwining isometries and unitaries that appear in the proofs of Corollary \ref{balance} and Theorem \ref{NCVRisometry2} must also satisfy the
 analogous condition on the norms of the residual compact operators.
\end{proof}

We are now ready to prove Theorem \ref{GequivControlledPPV}.

\begin{proof}[Proof of Theorem \ref{GequivControlledPPV}:]
Let $S_\ep\in \maL_{C(X)}(\ell^2\Gamma^\infty\otimes (H_1\oplus H_2)\otimes C(X), \ell^2\Gamma^\infty\otimes (H_2\oplus H_1)\otimes C(X)$ be a unitary, obtained from Corollary \ref{ControlledPPViso2},
such that
\begin{equation}\label{intertwine}
S_\ep^*\what{\pi}_2(f)S_\ep- \what{\pi}_1(f) \stackrel{\epsilon}{\sim} 0 \quad  \forall f\in \Sigma.
\end{equation}
Observe that if $f'\in C(X, A_c)$, with $A_c:=C_c(Z)A$, satisfies  $||f-f'||_\infty \leq \epsilon$, then the analogous relation to \ref{intertwine} also holds for $f'$,
and vice versa,  if the relation holds for $f'$ it also holds for $f$. Also note that since $\Sigma$ is globally $\Gamma$-invariant,
the construction of the $\Gamma$-invariant unitary $W_\ep$ which intertwines the representations $\what\pi_1$ and $\what\pi_2$ then follows from Theorem \ref{GequivVRiso2}, using the
norm-controlled operator $S_\ep$. The only thing to check is that for all $f\in \Sigma$,
$$
\what{\pi}_2(f)W_\ep- W_\ep\what{\pi}_1(f) \stackrel{\epsilon}{\sim} 0.
$$
Let $f\in \Sigma$, we first show that the required relation holds for any $f'\in C(X, A_c)$ such that $||f-f'||_\infty \leq \epsilon$.
First note that by the localization of the support of  $S_\ep$, for $f'\in C(X, A_c)$, the sum defining $\what\pi_2(f')W_\ep$ is again finite. Moreover, the
number of terms in the finite sum is independent of the operator $S_\ep$ itself, and therefore independent of $\epsilon$.

We have from the computations in the proof of Theorem \ref{GequivVRiso2}, for $f'$ as above,
$$
\what{\pi}_2(f')W- W\what{\pi}_1(f')=   \sum_{g\in \Gamma} (\id_{H_2\oplus H_1}\otimes U_g)\circ g{\left[\left(\what\pi_2(g^{-1}f')S-S\what\pi_1(g^{-1}f')\right) \pi_1(\sqrt{\chi})\right]}
$$
{This sum is again over a finite subset $\Gamma(\chi,f')$ of $\Gamma$,  due to the assumption of uniform proper action and given the support condition for  $S$.} As before, if we let $K_g(\chi, f')$ denote the compact operator $[\what\pi_2((g^{-1})^*f')S_\ep-S_\ep\what\pi_1((g^{-1})^*f')]$, we have $K_g(\chi, f')\stackrel{\epsilon}{\sim} 0$ for each $g \in \Gamma$,
say with the constant of inequality $C_g>0$, and hence we have
$$
||\what{\pi}_2(f')W- W\what{\pi}_1(f')||\leq \left(\max_{g\in \Gamma(\chi, f')}C_g\right). |\Gamma(\chi,f')|.\epsilon.
$$
Now as $||f-f'||\leq \epsilon$, we also have:
$$
\what{\pi}_2(f)W- W\what{\pi}_1(f)\stackrel{\epsilon}{\sim} 0
$$
{In the cocompact and metric-proper case, notice that $\Gamma(\chi,f')$ is contained in the set of $g\in \Gamma$ such that $\Supp((g^{-1})^*f')\cap \overline{B_\kappa(\Supp(\chi))}\neq \emptyset$, where $\kappa$ is the diameter of the cutoff function $\chi$.}
This ends the proof.
\end{proof}

We have now completed the proof of Theorem \ref{GequivVRiso}.
%
%\medskip
%
%\begin{theorem}\label{GequivControlledPPVseq}
% Under the assumptions of Theorem \ref{GequivVRiso2}, there exists a sequence $\{W_n\}_{n\in \N}$ of {\underline{$\Gamma$-invariant}}  unitary operators
%$$
%W_n\in \maL_{C(X)} \left([(H_1\oplus H_2)\otimes \ell^2\Gamma^\infty]\otimes C(X) ,  [(H_2\oplus H_1)\otimes \ell^2\Gamma^\infty] \otimes C(X)\right),
%$$
%with  localized support as in Theorem \ref{GequivVRiso2}, and such that for any $\varphi\in C(X, A)$,
%$$
%W^*_n\what\pi_2 (\varphi) W_n - \what\pi_1(\varphi)\text{ is compact, and }  \lim_{n\rightarrow\infty} ||W^*_n\what\pi_2 (\varphi) W_n - \what\pi_1(\varphi)||= 0.
%$$
%\end{theorem}

\medskip

\section{Application to equivariant Paschke duality}\label{PaschkeDuality}

As an application of our equivariant version of the PPV theorem, stated in Theorem \ref{GequivVRiso2}, we now prove the Paschke-Higson duality theorem in this context. We assume in this section that the proper $\Gamma$-space $Z$ is {\underline{cocompact}} and endowed as before with the $\Gamma$-equivariant metric $d$ so that closed balls are compact subspaces of $Z$, said differently the metric space $(Z, d)$ is proper. Recall that $A$ is a proper $\Gamma$-algebra over $Z$ and that we have assumed that $C_0(Z)$ maps {inside  $A$ itself}.
%\subsection{Some preliminaries}
Recall from \cite{BR2}, that associated with the proper
metric space $Z$ and a proper action of the groupoid $G=X\rtimes \Gamma$ on the $C^*$-algebra $C(X, A)$, we can define the $G$-equivariant Roe algebras, which will be denoted as
$D^*_\Gamma(X, A; \ell^2\Gamma^\infty\otimes H)$ and  $C^*_\Gamma(X, A; \ell^2\Gamma^\infty\otimes H)$ associated with a given ample $\Gamma$-equivariant representation of $A$ in $H$. The first one is  the closure of the space of pseudo-local $\Gamma$-invariant operators, while the second one is the ideal in the first one composed of those operators that are moreover locally compact. The
quotient algebra is denoted as $Q^*_\Gamma(X;(Z,\ell^2\Gamma^\infty\otimes L^2Z))$.

Let us recall the precise definitions  which are the immediate generalizations of the ones given in \cite{BR1} and \cite{BR2} when $A= C_0(Z)$. Let $(H, U)$ be a unitary Hilbert space representation of $\Gamma$ together with an ample $\Gamma$-equivariant representation $\pi$ of $A$.
%This is equivalent to the datum of a $G$-equivariant
%{$C(X)$}-representation $\hat\pi$ of {$C(X, A)$}.
%{Notice indeed that if $\hat\pi$ is given then $\pi (a):=\hat\pi (1_X\otimes a)$ will give the corresponding representation of $A$, and vice versa}.
%{inducing a so-called fibrewise ample representation of $C_0(Y)$ on $H\otimes C(X)$, see definition \ref{FiberwiseAmple}}.
Recall that any adjointable operator $T$ of $ \maL_{C(X)} (C(X)\otimes H)$ is given by a $*$-strongly continuous  field $(T_x)_{x\in X}$ of bounded operators on $H$.
An adjointable operator is  $G$-equivariant, if the field $(T_x)_{x\in X}$ satisfies the relations
$$
T_{x g} = U_g^{-1} T_x U_g, \quad  (x, g)\in X\times \Gamma.
$$
The space of $G$-equivariant adjointable operators is denoted as usual $\maL_{C(X)} (C(X)\otimes H)^\Gamma$. Recall also from the previous section the notion of propagation of a given operator with respect to the $C(X)$-$\Gamma$-equivariant representation $\hat\pi$ of $C(X, A)$.
We denote by $D^*_\Gamma (X, A; H)$ and $C^*_\Gamma (X, A; H)$ the corresponding Roe algebras as defined in \cite{BR1},
but for our groupoid $G$ and our specific Hilbert $G$-module $C(X)\otimes H$. More precisely,
$D^*_\Gamma (X, A; H)$ is defined as the norm closure in $\maL_{C(X)} (C(X)\otimes H)$ of the following space
$$
\{T\in \maL_{C(X)} (C(X)\otimes H)^G, T \text{ has finite propagation and }[T, \pi (f)]\in C(X, \maK(H)) \text{ for any }f\in C(X, A)\}.
$$
The ideal $C^*_\Gamma (X, A; H)$ is composed of all the elements $T$ of $D^*_\Gamma (X,A; H)$ which satisfy in addition that
$$
T \pi (f) \in C(X, \maK(H)) \text{ for any }f\in C(X, A).
$$
The finite propagation property here is supposed to hold uniformly on $X$, so $(T_x)_{x\in X}$ has finite propagation if there exists a constant $M\geq 0$ such that
for any $\varphi, \psi\in C(X, A)$ with $d (\Supp (\varphi), \Supp (\psi)) > M$, we have
$$
\pi_x (\varphi) T_x \pi_x (\psi) = 0, \; \; \forall x\in X.
$$
We thus have the short exact sequence of $C^*$-algebras
$$
0 \to C^*_\Gamma (X, A; H) \hookrightarrow D^*_\Gamma (X, A; H)  \longrightarrow Q^*_\Gamma (X, A; H) \to 0,
$$
where we have denoted by $Q^*_\Gamma (X, A; H) $  the quotient $C^*$-algebra of $D^*_\Gamma (X, A; H) $ by its two-sided closed involutive  ideal $C^*_\Gamma (X, A; H)$. The notation here is ambiguous as we don't mention the space $Z$ while the notion of propagation with respect to the representation of $A$ depends  a priori on the choice of $(Z, d)$.  The reason for this simplified notation is that the $K$-groups will not depend on this choice as we shall see below, although the identifications are not natural.

The Paschke-Higson duality theorem identifies the $K$-theory of the quotient algebra
$Q^*_\Gamma(X, A; \ell^2\Gamma^\infty\otimes H)$ with the $G$-equivariant $KK$-theory of the pair $C(X, A), C(X)$. For details about the definition of $G$-equivariant $KK$-theory the reader
is referred to the fundamental paper of Le Gall \cite{LeGall}. Since $X$ is compact here, notice though that the latter group  is naturally isomorphic to the $\Gamma$-equivariant $KK$-theory of the pair
$(A, C(X))$, see \cite{BR2}, section 4, for more details.

When $A=C_0(Z)$, we can for instance make use of the representation $\pi_{X\times Z}$ which is induced by multiplication on $\ell^2\Gamma^\infty\otimes L^2 (Z)\otimes C(X)$, where $L^2Z=L^2 (Z, \mu_Z)$ is defined
for a choice of a Borel $\Gamma$-invariant measure $\mu_Z$ on $Z$, which we shall always assume to be fully supported. This representation is fibrewise ample in the sense of Definition
\ref{FiberwiseAmple}.

%\subsection{Proof of the Paschke-Higson duality}
We are now in position to prove Theorem \ref{Paschke2}. We  need to construct a group isomorphism
$$
\maP_*: K_*(Q^*_\Gamma (X, A; \ell^2\Gamma^\infty\otimes H)) \xrightarrow{\cong} KK^\Gamma_{*+1}(A, C(X)), \quad \quad  *= 0, 1.
$$
 We only treat the case $*=0$. The proof is again a repetition of the proof  given in \cite{BR2}, Theorem 4.1, and adapted to the more general proper $\Gamma$-algebra $A$; we sketch it here only for completeness.
 We construct a group homomorphism $\maP'_0: KK^1_{\Gamma}(A, C(X))\rightarrow K_0(Q^*_\Gamma (X, A; \ell^2\Gamma^\infty\otimes H)) $, using the equivariant PPV Theorem \ref{GequivVRiso2}.
 The homomorphism $\maP'$ will then be an inverse to the natural Paschke-Higson map
 $$
 \maP_0: K_0(Q^*_\Gamma (X, A; \ell^2\Gamma^\infty\otimes H))  \rightarrow KK^\Gamma_1 (A, C(X))
 $$

 \textbf{Step 1:} Let $[(\sigma, E, F)]\in KK^\Gamma_1 (A, C(X))$. We may assume as usual that $\sigma$ is non-degenerate and that $F$ is self-adjoint.
 Using Kasparov's stabilization theorem, we obtain a cycle of the form $[\sigma_1, H \otimes C(X), F_1]$,
 which is endowed with the transported $G$-action via the Kasparov isomorphism $E\oplus (H\otimes C(X))\cong H\otimes C(X)$. Note that the summand $H\otimes C(X)$
 appearing on the left side of the isomorphism is endowed with its canonical $G$-action induced by the action of $G$ on $C(X, A)$. It is easy to check that the latter cycle
 lies in the same $KK_1^\Gamma$-class as $[\sigma, E, F]$.

 \textbf{Step 2:} Embed $H\otimes C(X)$ equivariantly in $\ell^2\Gamma\otimes H\otimes C(X)$ via an equivariant isometry
  $S: H\otimes C(X)\rightarrow  \ell^2(\Gamma)\otimes H\otimes C(X)$, defined by the following formula which uses the cut-off function $\chi \in C_c (Z)$ as in the previous section:
 $$
 S(e)= \sum_{g\in \Gamma} \delta_{{{g^{-1}}}}\otimes \sigma_1(g\sqrt{\chi})(e) \quad \text{ for }e\in \maE_c,
 $$
  where $\maE_c = \pi (A_c) (H\otimes C(X))$, the $G$-action on $H\otimes C(X)$ is given by the action $V$ from Step 1, while the $G$-action on $\ell^2(\Gamma)\otimes H\otimes C(X)$ is given by the
  right regular representation of $\Gamma$ on $\ell^2\Gamma$ tensored by the same action $V$.

  Now, $\left(S\sigma_1(\bullet)S^*, SS^*(\ell^2\Gamma\otimes H\otimes C(X)), SF_1S^*\right)$ is equivalent to $(\sigma_1, H\otimes C(X), F_1)$, and
  after adding a suitable degenerate cycle, we get the cycle $\left(\sigma_2:= \id_{\ell^2\Gamma}\otimes \sigma_1, \ell^2\Gamma\otimes H\otimes C(X), F_2:=(F_1\oplus \id)\right)$ which is
  still in the same $KK^\Gamma_1$-class as $(\sigma_1, H\otimes C(X), F_1)$. For details of this construction we refer the reader to \cite{BR2}, see Step 2 of the proof of Theorem 4.1 there.

  \textbf{Step 3}: Add further degenerate cycles to $[\sigma_2, \ell^2\Gamma\otimes H\otimes C(X), F_2]$ we may pass to a new $\Gamma$-equivariant Kasparov cycle
  $\left({{\sigma_2^\infty:=}}\id_{\ell^2\N}\otimes \sigma_2, \ell^2\Gamma^\infty \otimes H\otimes C(X), F_2^\infty:=\diag(F_2, \id, \id...)\right)$ which represents
  the same $KK^\Gamma_1$-class. We further add the degenerate cycle $\left(0,\ell^2\Gamma^\infty\otimes H\otimes C(X), 0\right)$
  to $[\sigma_2^\infty,  \ell^2\Gamma^\infty\otimes H\otimes C(X), F^\infty_2]$ with the $\Gamma$-action
  now taken as the one coming canonically from the $\Gamma$-action on {$H\otimes C(X)$} tensored with the right regular representation on the factor $\ell^2\Gamma$ and extended trivially
  on $\ell^2\N$. We obtain in this way  a new $\Gamma$-equivariant Kasparov cycle
  $$
  \left(\sigma_3:= \sigma^\infty_2\oplus 0,\ell^2\Gamma^\infty\otimes (H\oplus H)\otimes C(X), F_3:= F_2{{^\infty}}\oplus 0\right)
  $$
  still remaining in the same $KK^\Gamma_1$-class.

  \textbf{Step 4:}  We can now apply Theorem \ref{GequivVRiso2}, to get a $\Gamma$-invariant $C(X)$-adjointable unitary $W$ such that
  $$
  W \sigma_3(f)W^*- (\what\pi^\infty(f)\oplus 0) \in \maK_{C(X)}\left(\ell^2\Gamma^\infty\otimes (H\oplus H)\otimes C(X)\right),\quad  \text{ for all } f\in C(X, A).
  $$
 where $\pi^\infty: C(X, A)\rightarrow \maL_{C(X)}(\ell^2\Gamma^\infty\otimes H\otimes C(X))$ is induced by the ample representation $\pi$ of $A$ in the Hilbert module $H\otimes C(X)$
 and extended by the identity on $\ell^2\Gamma^\infty$. By Kasparov's homological equivalence Lemma (see \cite{BR2}, Appendix B), the cycles
  $$
  \left(\sigma_3,\ell^2\Gamma^\infty\otimes (H\oplus H)\otimes C(X), F_3\right)\quad \text{ and }\quad
 \left(\pi^\infty \oplus 0, \ell^2\Gamma^\infty\otimes (H\oplus H)\otimes C(X), F_4\right),
 $$
live in the same $KK^\Gamma_1$-class, where $F_4:= WF_3W^*$.

 \textbf{Step 5:} Let {$F_5$  be the $(1,1)$-entry in the $2\times 2$-matrix decomposition of $F_4$, corresponding to the direct sum
  $\ell^2\Gamma^\infty\otimes (H\oplus H)\otimes C(X)$ . Then
  the cycle
  $$
  [\what\pi^\infty \oplus 0, \ell^2\Gamma^\infty\otimes (H\oplus H)\otimes C(X), F_4]
  $$
  is in the same $KK^\Gamma_1$-class as the
  cycle $[\what\pi^\infty, \ell^2\Gamma^\infty\otimes H\otimes C(X), F_5]$}.

  \textbf{Step 6:} {Replace the operator $F_5$  by a $\Gamma$-invariant finite propagation operator $F_6$ as usual  by averaging $\sqrt{\chi}F_5\sqrt{\chi}$. We define the inverse map $\maP':  {KK^{\Gamma}_{1}(C(X, A), C(X))}\rightarrow K_0(Q^*_\Gamma (X, A; \ell^2\Gamma^\infty\otimes H)$ by setting
  $$
  \maP'([\sigma, E, F]):= {\left[q\left(\frac{1}{2}(\Id_{\ell^2\Gamma^\infty\otimes H\otimes C(X)} + F_6)\right)\right]}
  $$}
  where $q: D^*_\Gamma (X, A; \ell^2\Gamma^\infty\otimes H)\rightarrow Q^*_\Gamma (X, A; \ell^2\Gamma^\infty\otimes H)$ is the quotient projection.

  The map $\maP'$ is well-defined and a bijective group homomorphism, following the same arguments as in the compact case in \cite{BR2}, Theorem 4.1. Hence the proof of our Paschke-Higson theorem is now complete.

\appendix

\section{Localized operators on uniformly proper $\Gamma$-spaces}\label{LocalizedOperators}

{We prove in this appendix some standard results about supports of our localized operators that are used in some proofs.
Let us fix  a non-degenerate $*$-representation  $\pi: C_0(Z)\to \maL (H)$ of the $C^*$-algebra $C_0(Z)$ in the separable Hilbert space $H$, that we extend to $C_b(Z)$ as usual.
Recall that $\Gamma$ acts uniformly properly on $Z$ and that $\chi$ is a chosen {uniform} continuous cutoff function. }

{We shall use the following notations for an operator $T\in \maL(H)$}:
$$
\Supp (T)_z:=\{z'\in Z\vert (z', z)\in \Supp (T)\}, \quad \quad \Supp (T)^{z'}:=\{z\in Z\vert (z', z)\in \Supp (T)\}
$$

%Let us first state the following standard lemma.
%
%\begin{lemma}\label{PartialSupport} (PRELIMINARY, NOT TOTALLY SURE YET, WE ONLY NEED IT FOR LOCALIZED OPERATORS)
%{Let $T\in \maL(H)$ be a bounded operator and assume that we are given for $z\in Z$ a continuous bounded function $f\in C_b(Z)$ whose support does not intersect $\Supp (T)_z:=\{z'\in Z\vert (z, z')\in \Supp (T)\}$. Then there exists an open neighborhood $U$ of $z$ such that for any $\varphi \in C_0(U)$, we have $\pi (\varphi) T \pi (f)=0$. }
%\end{lemma}
%
%\begin{proof}
%{??? }
%\end{proof}
%
%{The same proof shows that given $g\in C_b(Z)$ whose support does not intersect  $\Supp (T)^z:=\{z'\in Z\vert (z', z)\in \Supp (T)\}$,  there exists an open neighborhood $V$ of $z$ such that for any $\psi\in C_0(V)$, we have $\pi (g) T \pi (\psi) =0$. }

Notice that if $W_\chi=\{\chi\neq 0\}$  then $
Z = \bigcup_{g\in \Gamma} g W_\chi.$
We  denote as in Section \ref{Statements} for any $k\geq 1$:
$$
\Gamma_\chi^{(k)} := \{(g, g')\in \Gamma^2\vert \exists (g_i)_{0\leq i\leq k-1} \text{ such that }g_iW_\chi\cap g_{i+1}W_\chi \neq \emptyset\text{ and }g_0=g, g_{k}=g'\}.
$$
For $k=0$, we set $\Gamma_\chi^{(0)}=\Gamma$ viewed as the diagonal of $\Gamma^2$.
Notice that $\Gamma_\chi^{(k)} \subset \Gamma_\chi^{(k+1)}$ for any $k$, and that $\bigcup_{k\geq 0} \Gamma_\chi^{(k)} = \Gamma^2$.
Recall that the uniform properness of the action means that the first (or the second) projection $\Gamma^2\to \Gamma$ becomes proper when restricted to $\Gamma_\chi^{(1)}$.
It is an obvious observation that if the proper $\Gamma$-space $Z$ is cocompact, then the action of $\Gamma$ on $Z$ is automatically uniformly proper since the support of {$\chi$} can then be taken compact, so that $\{g\in \Gamma\vert g\Supp (\chi)\cap \Supp (\chi) \neq \emptyset\}$ is finite.

Set $\maA_k := \bigcup_{(g, g')\in \Gamma_\chi^{(k)}} gW_\chi\times g'W_\chi$,
then it is easy to check using the properties of $W_\chi$  that for any $k\geq 0$ the closure of $\maA_k$ is contained $\maA_{k+2}$.
\begin{definition}[Localized operators]\label{LocalizedOp}
 {An operator $T\in \maL (H)$ is said to have \emph{localized support} if there exists $k\geq 0$ so that  $\Supp (T)$ is contained in (the closure of) some $\maA_k$ with $k\geq 0$.}

 The least $k$ such that the support of $T$ is contained in $\maA_k$ will be called the propagation index of $T$ (with respect to $\chi$). {For brevity, we shall call an operator with finite propagation index a \textit{localized} operator.}
\end{definition}

For a localized operator $T$ with propagation index $k$ and  if we denote by $\Gamma_z$ the finite subset of $\Gamma$ composed of those $g$ for which $z\in gW_\chi$, then for any $z\in Z$ we have:
$$
\Supp (T)_z\subset \bigcup_{g\in \Gamma_z} \;\; \bigcup_{g'\vert (g,g')\in \Gamma_\chi^{(k)}} g'W_\chi.
$$

{\begin{proposition}
Assume that $Z$ is a proper cocompact $\Gamma$-space with a $\Gamma$-invariant distance $d$ such that $Z$ is a metric-proper space. Then localized operators coincide with finite propagation operators.
\end{proposition}}

\begin{proof}
{We can find a cutoff function $\chi$ which is compacty supported in $Z$ and hence whose support has finite diamater. An operator $T$ is localized with propagation index $\leq k$ if and only if its support is contained in $\maA_k$.
%$$
%\maA_k = \bigcup_{(g, g')\in \Gamma_\chi^{(k)}} \; gW_\chi\times g'W_\chi.
%$$
Hence denoting by $d_\chi$ the diamater of $W_\chi$ in $Z$ which is equal to the diameter of any translate $gW_\chi$ for $g\in \Gamma$, we see that for any $(z, z')\in \Supp (T)$, we have by
$$
d(z, z') \leq k d_\chi.
$$
Hence $T$ has finite propagation $\leq kd_\chi$. If conversely $T$ has finite propagation $\kappa$. For any $(z, z')\in \Supp (T)$, we have $d(z, z')\leq \kappa$ and we also know that  there exists $g_1\in \Gamma$ such that $z\in g_1W_\chi$.  Since $Z$ is metric-proper, there exists  a finite subset $\Gamma_\kappa$ of $\Gamma$ such that the closed ball neighborhood $B_\kappa  :=\{z\in Z\vert d(z, \Supp (\chi))\leq \kappa\}$ of the compact space $\Supp (\chi)$ is contained in $\cup_{g\in \Gamma_\kappa} gW_\chi$. Moreover, let us denote by $k$ the least integer such that for any $g\in \Gamma_\kappa$, we have $(e, g)\in \Gamma_\chi^{(k)}$, with $e$ being the neutral element of $\Gamma$. To sum up we  know that  $z\in g_1W_\chi$ while $d(z, z')\leq \kappa$ so that $z'\in \bigcup_{g\in \Gamma_\kappa} \; g_1g W_\chi$, and henceforth
$$
(z, z')\in  \bigcup_{g\in \Gamma_\kappa} \; g_1W_\chi \times g_1 g W_\chi \subset \bigcup_{(g, g')\in \Gamma_\chi^{(k)}} \; gW_\chi \times g' W_\chi = \maA_k,
$$
and $k$ is of course independent of the chosen $(z, z')\in \Supp (T)$.}
\end{proof}

\begin{proposition}\
{The space  of localized operators is unital $*$-subalgebra of $\maL (H)$. Moreover,
\begin{enumerate}
\item the propagation index of the adjoint  is equal to the propagation index of the given localized operator.
\item the propagation index of the sum of two localized operators is $\leq$ to the maximum of the propagation indices.
\item the  propagation index of the composition of two localized operators  is $\leq 3 + $ the sum of the propagation indices.
\end{enumerate}}
\end{proposition}

\begin{proof}\
{{The first item is clear since the one has the relation $\Supp(T^*)= \sigma(\Supp(T))$, where $\sigma:Z\times Z\rightarrow Z\times Z$ is the involution $(z,z')\mapsto (z',z)$}. The support of the identity operator is the diagonal in $Z^2$ which is contained in $\Gamma_\chi^{(0)}$. Take two localized operators $T$ and $S$ with propagation indices $k$ and $k'$ respectively. The support of the sum $T+S$ is obviously contained in $\Supp (T)\cup\Supp (S)$.  Therefore
$$
\Supp (T+S) \subset \maA_k \cup \maA_{k'} = \maA_{\max(k,k')}.
$$
On the other hand, for any $(z,z'')$ such that  $\Supp (T)_z\cap \Supp (S)^{z''} \neq \emptyset$,  and denoting   the propagation index  of $T$ by $k$ and the propagation index of $S$ by $k'$, there exists $(g_0, \cdots, g_k)\in \Gamma^k$ and $(g'_0, \cdots, g'_{k'})\in \Gamma^{k'}$ such that
\begin{multline*}
z\in g_0W_\chi, z''\in g'_{k'}W_\chi, g_iW_\chi\cap g_{i+1}W_\chi\neq \emptyset \text{ for }0\leq i\leq k-1, \\ g_kW_\chi\cap g'_0W_\chi\neq \emptyset  \text{ and }g'_{j}W_\chi\cap g'_{j+1}W_\chi\neq\emptyset\text{ for }0\leq j\leq k'-1.
\end{multline*}
Hence $(z, z'')\in \maA_{k+k'+1}$. Hence,  using that the support of $TS$ is contained the closure of $\{(z, z'')\in Z^2\vert \Supp (T)_z\cap \Supp (S)^{z''} \neq \emptyset\}$ and the inclusion
$$
\overline{\maA_{k+k'+1}} \subset \maA_{k+k'+3}
$$
 we deduce that the support of $TS$ is contained in $\maA_{k+k'+3}$.}
%. Choose a continuous bounded function $f$  on $Z$ such that
%$$
%f= 0 \text{ on }\Supp (T)_{z}\text{ and }f= 1 \text{ on }\Supp (S)^{z''}.
%$$
%Using Lemma \ref{PartialSupport}, we can then find open neighborhoods $U$ and $U''$ of $z$ and $z''$ respectively, such that for any $\varphi\in C_0(U)$ and any $\varphi''\in C_0(U'')$ we have
%$$
%\pi (\varphi) T \pi (f) = 0 \text{ and } \pi (1-f) S\pi (\varphi'') = 0.
%$$
%Therefore
%we deduce
%$$
%\pi (\varphi) T S \pi (\varphi '') = \pi (\varphi) T \pi (f) S \pi (\varphi'') + \pi(\varphi T \pi (1-f) S\pi (\varphi'') = 0.
%$$
%Hence $(z, z'')\not\in \Supp (TS)$. Hence the support of $TS$ is contained in the closure of
%$$
%\{ (z, z'')\in Z^2\vert \Supp (T)_z\cap \Supp (S)^{z''} \neq \emptyset\}.
%$$
%{Now denoting the propagation index  of $T$ by $k$ and the propagation index $S$ by $k'$, we deduce that Now, as already observed the closure of this latter space is contained in }
\end{proof}

%{As a first corollary, we deduce that composition of two operators with $0$ propagation indices has  propagation index $\leq 3$.}

\begin{remark}\label{RoeLocalized}
{The analogously defined Roe $C^*$-algebras of locally compact and pseudolocal operators with \emph{localized support}, can hence be defined in our more general setting of non-cocompact uniformly proper actions.}
\end{remark}

\section{The norm-controlled PPV theorem}\label{ControlPPV}

{In this section we give a norm-controlled version of the PPV theorem \cite{PPV}[Theorem 2.10]. This is a folklore-type result which nevertheless is not found in the literature to the best of
our knowledge.}

Let $X$ be a finite dimensional compact metrizable space, $A$ a unital separable $C^*$-algebra and $H$ an infinite-dimensional separable Hilbert space. Denote $UCP(A, M_n)$ the space of unital, completely
positive maps from $A$ to $M_n(\C)$, equipped with the point-norm topology. {We shall denote by $\maL(H)_{*s}$ the algebra of bounded linear operators on $H$ equipped with the strong-$*$ topology}.

\begin{proposition}[Proposition 2.8 in \cite{PPV}]\label{PPViso1}
Consider an {exact} $X$-extension
$$
0\rightarrow C(X, K(H))\hookrightarrow B\xrightarrow{\sigma} A\rightarrow 0
$$
with ideal symbol $\{I_x\}_{x\in X}$ and $\Psi: X\rightarrow UCP(A, M_n)$ be a continuous map such that $\Psi(x)|_{I_x}=0$ for all $x\in X$. Then, given $\epsilon>0$, $V\subset H$ and
$1\in W\subset B$ finite-dimensional subspaces, there exists a norm-continuous map $U: X\rightarrow \maL(\C^n, H)$ such that
$$
U^*(x)U(x)= \id_{\C^n}, \quad \quad U(x)(\C^n) \perp V, \quad \forall x\in X
$$

and
$$
||\Psi(x)(\sigma(b))- U^*(x)b(x)U(x)||\leq \epsilon ||b|| \quad \quad \forall x\in X, b\in W.
$$

The linear span of $\{U(x)\C^n\}_{x\in X}$ in $H$ is finite-dimensional.
\end{proposition}

Using Proposition \ref{PPViso1}, one gets the following:

\begin{corollary}\label{PPVlim}
Consider an $X$-extension
$$
0\rightarrow C(X, K(H))\hookrightarrow B\xrightarrow{\sigma} A\rightarrow 0
$$
with exact  ideal symbol $\{I_x\}_{x\in X}$ and let $\Psi: X\rightarrow UCP(A, M_n)$ be a continuous map such that $\Psi(x)|_{I_x}=0$ for all $x\in X$. Let $V$ be a finite-dimensional subspace of $H$.
Then there exists a sequence of norm-continuous maps
$U_k: X\rightarrow  \maL(\C^n, H)$ such that
\begin{enumerate}
 \item $U^*_k(x)U_k(x)= \id_{\C^n} \quad U_k(x)(\C^n) \perp V \quad\quad \forall x\in X, \forall k\in \N$
 \item $\lim_{k\rightarrow \infty} \sup_{x\in X} ||\Psi(x)(\sigma(b))- U^*_k(x)b(x)U_k(x)|| = 0 \quad \quad \forall b\in B$, and
 \item $\lim_{k\rightarrow \infty} \sup_{x\in X} ||U_k^*(x) \eta(x) || =0 \quad  \forall \eta \in C(X,H)$.
 \item {The linear span of $\{U_k(x)\C^n\}_{x\in X}$ in $H$ is finite-dimensional for each $k \in \N$.}
\end{enumerate}

\end{corollary}

\begin{proof}
{Fix a convergent sequence $F_0=\{b_i\}_{i \in \N}$ in $B$ containing $1\in B$, such that $F_0=F_0^*$, $||b_i||\leq 1$, for all $i$, and the linear span of $F_0$ is dense in $B$.
{Recall that $B$ is separable here since it is an extension algebra}. Set $F:=\{b_\infty, b_\infty^*\} \cup F_0$ where $b_\infty$ is $\lim b_i$.}
Since $F$ is compact, for each $k\in \N$, there exists $N_k\in \N$ and a finite set $F_k:=\{b_{i_m}\}_{m=1}^{N_k}$ which includes $1\in B$,
such that for any $b_i\in F$ there exists an index $m\in \{1,2,..., N_k\}$ such that
$||b_{i_m}-b_i||<1/3k$.% and $F_{k-1}\subseteq F_k$ (we set $F_0=\emptyset$).

Let $\{e_n\}$ be an orthonormal basis for $H$. Denote by $P_j$ the linear span of $\{e_1, e_2, ..., e_j\}$.
For each $k\in \N$, we iteratively apply Proposition \ref{PPViso1} by taking $V_k= \span\{V, P_k\}$ and $W_k= F_k$ and $\epsilon_k= 1/3k$. We thus obtain norm-continuous maps $U_k: X\rightarrow \maL(\C^n, H)$ such that
$$
U^*_k(x)U_k(x)= \id_{\C^n} \quad U_k(x)(\C^n) \perp V_k \quad\quad \forall x\in X, \forall k\in \N
$$
which shows that (1) is satisfied. We also have

$$
\sup_{x\in X} ||\Psi(x)(\sigma(b))- U^*_k(x)b(x)U_k(x)|| \leq 1/3k \quad \quad \forall b\in W_k
$$

Now, for any $b\in F$, there exists an element $b'\in W_k$ such that $||b-b'||<1/3k$, then we get for any $x\in X$,
\begin{eqnarray*}
 ||\Psi(x)(\sigma(b))- U^*_k(x)b(x)U_k(x)|| &\leq & ||\Psi(x)(\sigma(b)-\Psi(x)(\sigma(b')|| + ||\Psi(x)(\sigma(b'))- U^*_k(x)(b')U_k(x)|| \\
 && + ||U_k^*(x)(b-b')U_k(x)|| \\
 &\leq & 1/3k+ 1/3k+1/3k = 1/k\\
 %&=& 1/k.[(||\Psi||+2)/3(1+||\Psi||)]\leq 1/k
\end{eqnarray*}
 where we have used the fact that $||\Psi||=1$ (since $A$ is unital) and $U_k(x)U_k^*(x)$ is an  {orthogonal} projection for all $x\in X$, so $||U_k(x)U^*_k(x)||{=} 1$. Thus (2) is established for all $b\in F$. Since $F$ spans $B$, another density argument then gives the result for all $b\in B$.

 To check (3), let $\epsilon>0$ and note that if $\eta \in P_j$ for some $j$, then $<\eta, U_k(x)U^*_k(x)\eta>=0$ for all $k>j$, since range of $U_k(x)$ is perpendicular to $P_k$.
 Now let $\eta=\sum_{i=1}^\infty \alpha_i e_i$, choose $N_0$ such that $||\eta-\sum_{i=1}^{N_0} \alpha_ie_i||<\epsilon$. Then for any $k$, we have
 $$
 ||U_k^*(x)\eta||\leq ||U_k^*(x)(\eta-\sum_{i=1}^{N_0} \alpha_ie_i)||+ ||U_k^*(x)(\sum_{i=1}^{N_0} \alpha_ie_i)||
 $$

 since $\sum_{i=1}^{N_0} \alpha_ie_i\in P_{N_0}$, the second term above is zero for $k>N_0$ for all $x\in X$. Therefore one gets
 $$
 \sup_{x\in X} ||U^*_k(x)\eta||\leq \epsilon \quad \quad \forall k>N_0, \forall x\in X.
  $$
 which establishes (3) in the case when $\eta\in C(X, H)$ is constant in the $X$-variable. To deal with the general case, let for each $x\in X$, $W_x$ be an open neighbourhood of $x$ such that for any $x'\in W_x$, we have:
 $$
 ||\eta(x)- \eta(x')||\leq \epsilon/2
 $$

 Since $X$ is compact we get a finite collection $\{W_{x_i}\}_{i=1}^m$ of such open neighbourhoods with centers $\{x_i\}_{i=1}^m$. Choose $N_0$ large enough such that
 $$
 \sup_{x\in X}||U_k^*(x)\eta(x_i)||\leq \epsilon/2 \text{   for all $k\geq N_0$ and for all } i=1,2,\cdots, m.
 $$

 Then we have for any $k\geq N_0$ and $x\in W_{x_i}$ for some $i$,
 $$
 ||U_k^*(x)\eta(x)||\leq ||U^*_k(x)(\eta(x)- \eta(x_i))||+||U^*_k(x)\eta(x_i)||\leq \epsilon
 $$
 This proves (3). {The last item (4) follows from the last line of Proposition (\ref{PPViso1}).}
\end{proof}

\begin{remark}
 In the proof above one can also take any countable approximate unit $\{a_k\}_{k\in \N}$ for $C(X, \maK(H))$ consisting of increasing sequence of finite-rank operators which are constant in $X$,
 and take $V_k= \span\{V, P_k\}$ where $P_k$ is the projection onto the range of $a_k$.
\end{remark}

Using the above result, we can now give a strengthening of Proposition 2.9 in \cite{PPV}. Denote by $d_x: C(X, \maL(H)_{*s})\rightarrow \maL(H)$ the evaluation map.
We keep the notations used above.

% \begin{theorem}\label{PPVseq}
%  Let $\mu: A\rightarrow C(X, \maL(H)_{*s})$ be a unital $*$-homomorphism such that $\ker(d_x\circ \mu)= I_x$ for all $x\in X$. Given $\epsilon>0$, there exists an isometry $S \in C(X, \maL(H)_{s^*})$ such that
%  \begin{enumerate}
%   \item $S^*bS- \mu(\sigma(b))\in C(X, \maK(H))$ \text{ for all } $b\in B$.
%   \item {{$\forall b\in B$, $\exists C$ independent of $\epsilon$ such that $||S^*bS- \mu(\sigma(b))||\leq C \epsilon$.}}
%  \end{enumerate}
%
% \end{theorem}

\begin{theorem}\label{PPVseq}
{{Given a trivial $X$-extension by $A$ with exact lsc ideal symbol $\{I_x\}_{x\in X}$:
$$
0\rightarrow C(X, \maK(H))\hookrightarrow B_1\xrightarrow {\sigma_1} A\rightarrow 0
$$
which is implemented by a unital $*$-homomophism $\mu_1: A\rightarrow C(X, \maL(H)_{*s})$ and another arbitrary $X$-extension with same ideal symbol $\{I_x\}_{x\in X}$, whose
extension algebra is $B\subseteq C(X, \maL(H)_{*s})$ for some infinite-dimensional separable Hilbert space $H$:
$$
0\rightarrow C(X, \maK(H))\hookrightarrow B\xrightarrow {\sigma} A\rightarrow 0
$$
Let $F$ be a compact subset of $B$ such that $F=F^*$, $1\in F$ and the linear span of $F$ is dense in $B$. Given $\epsilon>0$,
there exists an isometry $S \in C(X, \maL(H)_{s^*})$ such that}}
\begin{enumerate}[label=({\roman*})]
   \item ${bS- S\mu_1(\sigma(b))}\in C(X, \maK(H))$ \text{ for all } $b\in B$.
   \item {$\forall b\in F$, $\exists C$ independent of $\epsilon$ such that  $||S\mu_1(\sigma(b))- bS||\leq C\epsilon$.}
\end{enumerate}
\end{theorem}

\begin{proof}
{Recall that $B_1$ be the unital $C^*$-algebra generated by the image of $\mu_1$ and $C(X, \maK(H))$. Let} $\{a_k\}_{k=0}^\infty$ be a quasi-central approximate unit for $C(X, \maK(H))$ consisisting of an increasing sequence of constant (in the $X$-variable) finite-rank operators
$0=a_0\leq a_1\leq a_2\cdots$, $||a_k||\leq 1$, and
$$
\lim_k|| a_kl-l||=0,  \forall l\in C(X,\maK(H))\quad \text{and}\quad \lim_k ||[a_k,h]||=0,  \forall h\in B_1.
$$
where $[x,y]$ denotes the commutator $xy-yx$. Let $F$ be a compact, self-adjoint subset of the unit ball of $B$ whose span is $B$. Passing to a subsequence if necessary, we may assume that
$$
||[\mu_1(\sigma(b)), (a_k-a_{k-1})^{1/2}]||\leq \epsilon/2^k \quad \quad \quad \forall b\in F, k\geq 1.
$$
Let $Q_k$ be the constant orthogonal projection onto the range of $a_k$ for each $k\geq 1$. Using Corollary \ref{PPVlim},
we iteratively define a sequence of compact operators $U_k \in C(X, \maK(H)), k\in \N$ whose initial projections are the range of $a_k$ and final projections are of uniformly finite rank,
and an increasing sequence of finite-rank projections $R_k, k\in \N$ on $H$,
converging strongly to the identity, such that we have for all $k\geq 1$:
\begin{enumerate}
 \item $U_k^*(x)U_k(x)=Q_k$, for all $x\in X$.
 \item $\Range({U_k(x)})\subseteq (R_{k+1}-R_k)(H)$, for all $x\in X$.
 \item $\Range(U_k(x))\perp \Range(U_{k'}(x'))$ for all $x, x'\in X$ for all $k'<k$.
 \item $||{Q_k}\mu_1(\sigma(b))(x){Q_k} - U_k^*(x)b(x)U_k(x)||\leq \epsilon/2^k$, for all $x\in X$, $b\in F$.
 \item $||U_i^*(x)b(x)U_j(x)||\leq \epsilon/2^{i+j}$, for all $b\in F$, $x\in X$, $i\neq j$.
\end{enumerate}

Some remarks are in order. The first property is clear from the construction in Corollary \ref{PPVlim}; the existence of the finite-rank operators $R_k$ in the property (2)
also follows from the fact that the $U_k$ themselves are of uniformly finite-rank. The third property can be obtained in the construction of $U_k$ by adding the ranges of all the $U_{k'}$ for
$k'<k$ in the choice of the finite-dimensional space $V$ in Corollary \ref{PPVlim}. The fourth property is simply obtained by taking the completely positive map $\Psi$ in Corollary \ref{PPVlim}
to be ${Q_k}\mu_1(\bullet){Q_k}$. The last property (5) can be obtained from item (3) in Corollary \ref{PPVlim}, since the initial space of each $U_j$ for $j<i$ is of
uniformly finite-dimension.

Define the operator $S\in C(X, \maL(H)_{*s})$ pointwise in the following way:
$$
S(x):= \sum_{k=1}^\infty U_k(x)(a_k-a_{k-1})^{1/2}
$$
Indeed, it suffices to use properties (1), (2), and (3) above to show that $S(x)$ is uniformly convergent in $X$ with respect to the strong-$*$ topology on $\maL(H)$.
It can also be verified directly that $S^*(x)S(x)=\id_H$, thus $S(x)$ is an isometry, using the fact that $\Range(a_k-a_{k-1})^{1/2}\subseteq \Range(Q_k)=\Range(U_{k}^*(x)U_k(x))$ for
all $x\in X$.

Let $f_k= (a_k-a_{k-1})^{1/2}$. Using the fact that $\mu_1(\sigma(b))= \sum_{k=1}^\infty \mu_1(\sigma(b))f_k^2$, where the series converges in the strict topology, we get:
\begin{equation}\label{mustrict}
\mu_1(\sigma(b))(x)- \sum_{k=1}^\infty f_k \mu_1(\sigma(b))(x) f_k= \sum_{k=1}^\infty [\mu_1(\sigma(b))(x), f_k] f_k
\end{equation}
Thus, by the assumptions on $f_k= (a_k-a_{k-1})^{1/2}$, we get $||\mu_1(\sigma(b))(x)- \sum_{k=1}^\infty f_k \mu_1(\sigma(b))(x) f_k||\leq \epsilon$,
for all $b\in F$.

Therefore, we finally get for all $b\in F$,
\begin{multline}
||{S^*(x)b(x)S(x)}- \mu_1(\sigma(b))(x)||\leq  ||\mu_1(\sigma(b))(x)- \sum_{k=1}^\infty f_k \mu_1(\sigma(b))(x) f_k|| \\+ \sum_{k=1}^\infty ||{f_k\left(\mu_1(\sigma(b))(x)- U_k^*(x)b(x)U_k(x)\right) f_k}|| +\sum_{i\neq j} ||U_i^*(x)b(x)U_j(x)|| \leq 3\epsilon
\end{multline}
by properties (4) and (5) above, { and noticing that $Q_kf_k = f_kQ_k = f_k$}.

%{M2I: But here a modification in the argument is needed to apply property (4). Maybe modify that property accordingly?}

On the other hand, we also get $\mu_1(\sigma(b))- S^*b S \in C(X,\maK(H))$ for all $b\in F$. {Indeed, writing the expression pointwise as follows:}
\begin{eqnarray*}
\mu_1(\sigma(b))(x)- S^*(x)b(x) S(x) &=& \mu_1(\sigma(b))(x)- \sum_{k=1}^\infty f_k \mu_1(\sigma(b))(x) f_k\quad (I_1)\\
&+& \sum_{k=1}^\infty f_k\left(\mu_1(\sigma(b))(x)- U_k^*(x)b(x)U_k(x)\right) f_k\quad (I_2)\\
&+& \sum_{i\neq j} f_iU_i^*(x)b(x)U_j(x)f_j\quad (I_3)\\
\end{eqnarray*}

{Using the fact that $f_k$ and $U_k(x)$ are uniformly finite-rank for each $k$, we deduce the compactness of each term in the equation above as follows. The first term $I_1$ is compact due to equation (\ref{mustrict}), whose right hand side is compact since it is norm-convergent and each of its partial sum is compact. Similarly, $I_2$ and $I_3$ are compact because the partial sums for each series are all compact and each series is norm-convergent. Since the norm continuity in the $x$-variable of the left-hand side of the above equation is clear from the construction, we conclude that $\mu_1(\sigma(b))- S^*b S \in C(X,\maK(H))$ for all $b\in F$.}

Using a density
argument as before, one can finish the proof of $(i)$ by establishing the desired properties for all $b\in B$.

{Notice also that we have the following relation for any $b\in B$,
 \begin{multline}
 (S\mu_1(\sigma(b))- bS)^*(S\mu_1(\sigma(b))- bS)= (S^*b^*bS- \mu_1(\sigma(b^*b)))+\mu_1(\sigma(b^*))(\mu_1(\sigma(b)- S^*bS) \\+ (\mu_1(\sigma(b^*))- S^*b^*S)\mu_1(\sigma(b))
 \end{multline}
 from which the claim follows. }
\end{proof}

{\begin{remark}
{set for $b\in B$, $K_1(b):=bS-S\mu_1(\sigma (b))$ then  the operator $S\in C(X, \maL(H)_{*s})$ constructed in the proof above also satisfies
$$
K_1(b) S^* - S K_1(b^*)^* = [b, SS^*].
$$
{Hence $ [b, SS^*] \stackrel{\epsilon}{\sim} 0$ for all $b\in F$}.}
%
% $$
% ||b(1-SS^*)||= ||K_1^*S- K_1S^*||\leq 2C'\epsilon\; \text{ where }K_1= bS-S\mu_1(\sigma(b)).
% $$
\end{remark}}
%
%\begin{notations}
%For an operator $A\in C(X, \maL(H)_{*s})$, we shall use the notation $A\stackrel{\epsilon}{\sim} 0$
%to denote the fact that $A$ is compact and has norm at most a constant multiple of $\epsilon$; the constant may depend on $A$.
%\end{notations}

{{{We may rewrite for $b\in F$},  the relations in Theorem \ref{PPVseq}, using as well the previous remark, as
$$
{S^*bS-\mu_1(\sigma(b)) \stackrel{\epsilon}{\sim} 0, bS- S\mu_1(\sigma(b))\stackrel{\epsilon}{\sim} 0,
%b(1-SS^*)\stackrel{\epsilon}{\sim} 0,
[b,SS^*]\stackrel{\epsilon}{\sim} 0.}
$$
%Notice that $S^*bS-\mu_1(\sigma(b))$ satisfies conditions (1) and (2) in Theorem \ref{PPVseq}.
}}

\begin{corollary}\
{{{Let $H_1$ be a separable infinite-dimensional Hilbert space}. Given a trivial $X$-extension by $A$ with exact lsc ideal symbol $\{I_x\}_{x\in X}$:
$$
0\rightarrow C(X, \maK(H_1))\hookrightarrow B_1\xrightarrow {\sigma_1} A\rightarrow 0
$$
which is implemented by a unital $*$-homomophism $\mu_1: A\rightarrow C(X, \maL(H_1)_{*s})$ {{and another}} arbitrary $X$-extension with same ideal symbol $\{I_x\}_{x\in X}$, whose
extension algebra is $B\subseteq C(X, \maL(H)_{*s})$ for some infinite-dimensional separable Hilbert space $H$:
$$
0\rightarrow C(X, \maK(H))\hookrightarrow B\xrightarrow {\sigma} A\rightarrow 0,
$$
 there exists a sequence of operators $\{S_n\}_{n\in \N}$,
$S_n\in C(X, \maL(H_1, H)_{*s})$ for all $n\in \N$, such that we have for all $b\in B$:
\begin{enumerate}
 \item {$S_n\mu_1(\sigma(b)) - b S_n \in C(X, \maK(H_1))$ for any $n\in \N$,}
 \item {$\lim_{n\rightarrow \infty} || S_n\mu_1(\sigma(b)) - b S_n||=0$}, and
 \item $S_n^*S_n= \id_{H_1\otimes C(X)}$ for all $n\in \N$.
\end{enumerate}}}
\end{corollary}

\begin{proof}
 This is an immediate application of Theorem \ref{PPVseq}, by reducing to the case $H_1=H$ via a unitary isomorphism $u:H_1\rightarrow H$.
 %the condition (3) can be obtained by considering the sequence $S_n$ obtained from the theorem and considering the sequence
% $S_n':= S_{n+n_0} (S^*_{n+n_0}S_{n+n_0})^{-1/2}$, where $n_0\in \N$ is such that $||S_n^*S_n-\id||<1$ for all $n\geq n_0$.
\end{proof}

\begin{corollary}\label{PPVunitaryseq}
Given a trivial $X$-extension by $A$ with exact lsc ideal symbol $\{I_x\}_{x\in X}$:
$$
0\rightarrow C(X, \maK(H))\hookrightarrow B_1\xrightarrow {\sigma_1} A\rightarrow 0
$$
which is implemented by a unital $*$-homomophism $\mu_1: A\rightarrow C(X, \maL(H)_{*s})$, {{and another}} trivial
 $X$-extension with the same ideal symbol $\{I_x\}_{x\in X}$, {implemented similarly by $\mu$ and} whose
extension algebra is $B\subseteq C(X, \maL(H)_{*s})$:
$$
0\rightarrow C(X, \maK(H))\hookrightarrow B\xrightarrow {\sigma} A\rightarrow 0,
$$
{there exists for any $\ep >0$ and any compact subset $F$ in $A$ such that $F=F^*$, $1\in F$ a \emph{unitary} operator $S_\ep \in C(X, \maL(H)_{*s})$ such that we have
\begin{enumerate}
 \item $\mu_1(a) - (S_\ep)^* \mu (a)  S_\ep \in C(X, \maK(H))$ for any  $a\in A$;
 \item The family $S=(S_\ep)$ satisfies that for any $a\in F$, $\mu_1(a) - (S_\ep)^* \mu (a)  S_\ep \stackrel{\epsilon}{\sim} 0$.
\end{enumerate}}

\end{corollary}

\begin{proof}
This is done by the usual PPV trick to pass from isometries to unitaries, as in Theorem 2.10 in \cite{PPV}.
{Let us give some details for the convenience of the reader. Starting with any trivial $X$-extension, one obtains a new trivial $X$-extension with same ideal symbol, by amplifying as follows. Consider for instance the trivial $X$-extension given by $\sigma_1$, and set $H':=H\otimes \ell^2\N$ and the $X$-extension associated with the unital $*$-monomorphism
$$
\mu'_1: A\rightarrow C(X, \maL (H')_{*s})\text{ given by } \mu'_1 (a):=\mu_1 (a)\otimes \id_{\ell^2\N}.
$$
So, the $X$-extension is $0\rightarrow C(X, \maK(H'))\hookrightarrow B'_1\xrightarrow {\sigma'_1} A\rightarrow 0$ with $B'_1=\mu'_1(A) + C(X, \maK(H'))$. Applying the previous corollary, we deduce the existence of a sequence $(S^0_n)_{n\in \N}$ of isometries in $C(X, \maL(H', H)_{*s})$ such that
$$
S^0_n \mu'_1 (\sigma (b)) - b S^0_n \in C(X, \maK (H', H)) \text{ and } \lim_{n\to \infty}  ||S^0_n  \mu'_1(\sigma(b)) - b S^0_n||=0.
$$
As in \cite{PPV}[page 71], we then construct out of each isometry $S^0_n$ a unitary $U^0_n\in C(X, \maL(H, H {\oplus} H)_{*s})$ which intertwines the original $X$-extension given by $B$ and its direct sum with the one given by $B_1$. More precisely, we consider the shift isometry $V$ on $H'$ that we view as constant in the $X$-variable, and similarly  the projection $P:H'\to H$ onto  the first component seen as a constant co-isometry. It is an obvious observation that $PV=0$, hence also $V^*P^*=0$, and  we can define following again \cite{PPV} the unitaries $(U^0_n\in C(X, \maL (H, H\oplus H)_{*s}))_{n\in \N}$ by setting
$$
U^0_n := \left(\begin{matrix} \id - S^0_n (S^0_n)^* + S^0_n {V^*} (S^0_n)^*  \\ P (S^0_n)^* \end{matrix}\right).
$$
A straightforward computation using the properties of the family $(S^0_n)_n$, shows that
$$
 \left(\begin{matrix} b & 0 \\ 0 & \mu_1 (\sigma (b))\end{matrix} \right) U^0_n - U^0_n b \in C(X, \maK (H, H\oplus H)),\quad \forall b\in B.
$$
and that  $\left\vert\left\vert \left(\begin{matrix} b & 0 \\ 0 & \mu_1 (\sigma (b))\end{matrix} \right) U^0_n - U^0_n b\right\vert\right\vert \to 0$. }

{Applying the same construction  using now that the original $X$-extension given by $B$ is also $X$-trivial, one deduces a sequence $(U^1_n)_{n\in \N}$ of unitaires in $C(X, \maL(H, H\oplus H)_{*s})$ such that  (applying as well the unitary interchanging the two copies of $H$):
$$
 \left(\begin{matrix} \mu (\sigma_1 (b_1)) & 0 \\ 0 & b_1\end{matrix} \right) U^1_n - U^1_n b_1 \in C(X, \maK (H, H\oplus H)),\quad \forall b_1\in B_1.
$$
and that  $\left\vert\left\vert \left(\begin{matrix} \mu (\sigma_1 (b_1)) & 0 \\ 0 & b_1\end{matrix} \right) U^1_n - U^1_n b_1\right\vert\right\vert \to 0$. {Applying the above properties to $b=\mu (a)$ and $b_1=\mu_1(a)$ for $a\in A$, we get
$$
U_n^0 \mu(a) (U_n^0)^* - U_n^1 \mu_1(a) (U_n^1)^* \in C(X, \maK(H))\text{ and } \vert\vert U_n^0 \mu(a) (U_n^0)^* - U_n^1 \mu_1(a) (U_n^1)^*\vert\vert \to 0.
$$
Hence the conclusion using the unitaries $S_n:=(U^0_n)^* U^1_n$. It is then easy to check that the argument gives the same estimates over the compact subspace $F$.}
}
%It only remains to note that the condition (2) above is still valid; this can be easily verified by a direct inspection.
\end{proof}

\end{document}